\documentclass[12pt]{article}
\usepackage{amsmath}
\usepackage{ntheorem}
\usepackage{graphicx}
\usepackage{epsfig}
\usepackage{epstopdf}
\usepackage{natbib}
\usepackage{url} 
\usepackage{hyperref}
\usepackage{amsfonts}
\usepackage{amsmath,array}
\usepackage{booktabs}
\usepackage{natbib}
\usepackage{amssymb}
\usepackage{graphicx}
\usepackage{algorithm}
\usepackage{algorithmic}
\usepackage{cases}
\usepackage{epstopdf}
\usepackage{float}
\usepackage{geometry, tabularx}
\usepackage{multirow,makecell}
\usepackage{enumerate}
\usepackage{bm}
\usepackage{pifont}
\usepackage{footnote}
\usepackage{float}
\usepackage{titlesec}
\usepackage[titletoc]{appendix}
\usepackage{threeparttable}
\usepackage{setspace}
\usepackage{makecell}
\usepackage{mathrsfs}
\usepackage{color,soul}
\newcommand{\blind}{0}

\newtheorem{thm}{\bf Theorem}      
\newtheorem{lem}{\bf Lemma}

\newtheorem{rem}{\bf Remark}

\newtheorem*{proof}{Proof}

\begin{document}

\def\spacingset#1{\renewcommand{\baselinestretch}%
{#1}\small\normalsize} \spacingset{1}


\if0\blind
{
  \title{\bf  Partition-Insensitive Parallel  ADMM Algorithm  for High-dimensional  Linear Models}
  \author{Xiaofei Wu \hspace{.2cm}\\
    College of Mathematics and Statistics, Chongqing University, \\ Jiancheng Jiang \hspace{.2cm}\\ Department of Mathematics and Statistics, University of North Carolina at Charlotte, \\
    and \\
    Zhimin Zhang \thanks{Corresponding author. Email: zmzhang@cqu.edu.cn.
     The research of Zhimin Zhang was supported by the National Natural Science Foundation of China [Grant Numbers 12271066, 12171405, 11871121], and the research of Xiaofei Wu was supported by the project of science and technology research program of Chongqing Education Commission of China [Grant Numbers KJQN202302003].}\\
    College of Mathematics and Statistics, Chongqing University.}
  \maketitle
} \fi

\if1\blind
{
  \bigskip
  \bigskip
  \bigskip
  \begin{center}
    {\LARGE\bf Title}
\end{center}
  \medskip
} \fi

\bigskip
\begin{abstract}
The parallel alternating direction method of multipliers (ADMM) algorithms have gained popularity in statistics and machine learning due to their efficient handling of large sample data problems.  However, the parallel structure of these algorithms, based on the consensus problem, can lead to an excessive number of auxiliary variables when applied to high-dimensional data, resulting in {large computational burden}. In this paper, we propose a partition-insensitive parallel framework based on the linearized ADMM (LADMM) algorithm and apply it to solve nonconvex penalized  high-dimensional regression problems.
Compared to existing parallel ADMM algorithms, our algorithm does not rely on the consensus problem, resulting in a significant reduction in the number of variables that need to be updated at each iteration. 
It is worth noting that the solution of our algorithm remains  largely unchanged regardless of how the total sample is divided,  which is  known as partition-insensitivity. Furthermore, under some mild assumptions, {we prove the convergence of the iterative sequence generated by our parallel algorithm}. Numerical experiments on synthetic and real datasets demonstrate the feasibility and validity of the proposed algorithm. We provide a publicly available R software package to facilitate the implementation of the proposed algorithm.

\end{abstract}

\noindent%
\vfill

\newpage
\spacingset{1.5} 
\section{Introduction}
High-dimensional linear regression models can be written as  the form of `` loss + penalty ", 
\begin{align}\label{hdlr}
\mathop {\arg \min }\limits_{\bm \beta \in \mathbb{R}^p} \quad \sum_{i=1}^{n} \mathcal{L} (y_i - \bm x_i^\top \bm \beta)  + P_\lambda(|\bm \beta|),
\end{align}
{where $\bm{x}_i \in \mathbb{R}^p$ and $y_i \in \mathbb{R}$ denote the $i$-th observation values of the covariates and the response, respectively.}
The loss function $\mathcal{L}: \mathbb R \to \mathbb R^+$ is a versatile function that encompasses various forms, such as least squares, asymmetric least squares (\cite{Newey1987Asymmetric}), Huber loss (\cite{Huber1964Robust}), quantile loss (\cite{Koenker1978Regressions}), and smooth quantile loss (\cite{Jennings1993An} and \cite{Aravkin2014Sparse}). 
 \( P_\lambda(|\bm \beta|) \) represents various nonconvex penalty terms  including SCAD (\cite{Fan2001Variable}), MCP (\cite{Zhang2010Nearly}), Capped-$\ell_1$ (\cite{Zhang2010Analysis}), Snet (\cite{Wang2010Variable}), Mnet (\cite{Huang2016THE}), among others.

Over the past decade, many parallel ADMM algorithms have been proposed to solve (\ref{hdlr}) when dealing with particularly large samples ($n$ is  large) and/or distributed data storage. Examples include \cite{Boyd2010Distributed}, \cite{Yu2017ADMM}, \cite{Yu2017A}, and \cite{Fan2021Penalized}. The parallel algorithm is based on the principle of decomposing a complex optimization problem into smaller and more manageable subproblems that can be solved independently and simultaneously on local machines. Assuming there are \( M \) local machines available, the data \( \bm y = (y_1,y_2,\dots,y_n)^\top\) and \( \bm X  =  (\bm x_1,\bm x_2,\dots,\bm x_n)^\top \) can be divided into \( M \) blocks as,
\begin{align}\label{datasplit}
\bm y = (\bm y_1^\top, \bm y_2^\top, \dots, \bm y_M^\top)^\top, \ \bm X =(\bm X_1^\top, \bm X_2^\top, \dots, \bm X_M^\top)^\top.
\end{align}
Here, $\bm y_m \in \mathbb{R}^{n_m}$, $\bm{X}_m \in \mathbb{R}^{n_m \times p}$, and $\sum_{m=1}^{M} n_{m} =n$. In order to adapt to this parallel structure, the existing parallel ADMM algorithms employ strategies similar to the regularized consensus problem described by \cite{Boyd2010Distributed}. The regularized consensus problem is a method in machine learning and optimization for solving consensus problems in a distributed environment. It aims to achieve cooperation among distributed machines by running local optimization algorithms with regularization terms to promote model consistency and consensus.

By introducing $\{\bm r_m = \bm y_m - \bm{X_m \beta_m}\}_{m=1}^{M}$ and the consensus constraints $\{\bm \beta = \bm \beta_m\}_{m=1}^{M}$, the constrained optimization problem of (\ref{hdlr}) can be expressed as, 
\begin{align}\label{PCNPQR}
\min_{\bm \beta, \{\bm r_m, \bm \beta_m\}_{m=1}^M} & \quad   \sum_{m=1}^{M} \mathcal{L}(\bm r_m) + {P}_\lambda(|\bm \beta|), \notag \\
\text{s.t.} \ \bm X_m \bm \beta_m + \bm r_m & = \bm y_m, \  \bm \beta = \bm \beta_m, \  m =1,2,\dots,M,
\end{align}
where $\bm r_m = (r_{m,1},r_{m,2},\dots,r_{m,n_m})^\top$ and $\mathcal{L}(\bm r_m) =  \sum_{l=1}^{n_m} \mathcal{L}( r_{m,l}) $. 
To solve the constrained optimization problem in (\ref{PCNPQR}), both the central and local machines must handle subproblems of the same dimension as \( \bm \beta \) (which is \( p \) dimensions) in each iteration. However, in the case of penalized regression, it is common for $\bm \beta$ to have a relatively large dimension. As $\bm r = (\bm r_1^\top, \bm r_2^\top, \dots, \bm r_m^\top)^\top$ and each local machine needs to solve its corresponding $\bm \beta_m$, solving parallel problems results in solving at least $Mp$ subproblems more than the nonparallel case.

Although this phenomenon can be mitigated through parallel computing structures, where $M$ local machines compute $Mp$ subproblems in an equivalent time to a single machine computing $p$ subproblems, the presence of these additional auxiliary variables $\{\bm \beta_m\}_{m=1}^{M}$ can negatively impact the convergence speed and accuracy of the algorithm. This drawback has also been emphasized by \cite{Lin2022Alternating}, which suggested avoiding excessive auxiliary variables when designing the ADMM algorithm. Moreover, it is important to note that $\{\bm \beta_m\}_{m=1}^{M}$ is not the desired estimation value we initially aim for, but rather auxiliary variables added to avoid the need for double loops and facilitate a parallel structure. More details can be found in \cite{Yu2017A}. Hence, a natural question arises: Is it possible to design an ADMM algorithm without introducing these auxiliary variables, while also maintaining a parallel structure and avoiding double loops?

This paper proposes a new parallel strategy based on the LADMM to address the above problem. Compared to existing related algorithms, the proposed algorithm offers three main advantages:
\begin{itemize}
\item  \textbf{Not an approximation algorithm}: Our algorithm directly tackles nonnonvex optimization problems without relying on methods like local linear approximation (LLA) in \cite{Zou2008One},  majorization minimization step in \cite{Yu2017A}, or iterative reweighting in  \cite{Pan2021Iteratively}. For smooth quantile losses, our algorithm  directly derives its proximal operator, rather than using first-order approximation methods  in \cite{Mkhadri2017A}.  Additionally, we prove that our algorithm converges globally to  a critical point of the  nonconvex optimization problem.

\item \textbf{Fewer iterative variables}: Compared to existing parallel ADMM algorithms, our parallel LADMM algorithm requires fewer iterative variables. Existing algorithms typically handle subproblems with dimensions of at least \( (2M+1)p+2n \), while our algorithm only requires \( p+2n \). This reduction in the number of variables is especially significant when \( p \) and/or \( M \) are large.

\item \textbf{Partition-Insensitive}: As stated by \cite{Fan2021Penalized}, when more local machines are employed, the parallel algorithms proposed by \cite{Yu2017A} and \cite{Fan2021Penalized} tend to select more non-zero variables. In contrast, our algorithm ensures that the solution remains largely consistent across different numbers of local machines and the number of samples they process. This consistency guarantees the reliability and accuracy of the solutions, regardless of the chosen parallelization strategy.
\end{itemize}

To ensure the proposed algorithm's wide applicability, we opt for a highly flexible smooth quantile loss, denoted as $\mathcal{L}$ in (\ref{hdlr}), and employ general nonconvex penalties such as Snet and Mnet for $P_\lambda(|\bm \beta|)$ in (\ref{hdlr}).
The remainder of this paper is organized as follows.  In Section \ref{sec2}, we introduce two smooth quantile loss functions and some nonconvex penalties, and derives the closed-form solutions of their proximal operators. In Section \ref{sec3}, we propose the LADMM algorithm for solving nonconvex penalized regression problems  and derive its parallel version. Besides, we  explained the equivalence relationship between parallel and nonparallel algorithms and  provide the convergence of the algorithms in this section. In Section \ref{sec4}, we present the performance of the proposed algorithm on some synthetic datasets, and compare it with existing methods. In Section \ref{sec5}, we utilize the new parallel algorithm to conduct  regressions on an online publicly available dataset, leading to the discovery of several new findings. In Section \ref{sec6}, we summarize the findings and concludes the paper with a discussion of future research directions. The technical proofs and supplementary numerical experimentsis are included in the Appendix. R package PIPADMM for implementing the partition-insensitive parallel LADMM is available at \url{ https://github.com/xfwu1016/PIPADMM}.


\section{Preliminary}\label{sec2}
In this section, we introduce two smooth quantile loss functions and several nonconvex penalties, and combine them to create the generalized nonconvex penalized regression models. To facilitate the description of the algorithm, we  provide closed-form solutions for the proximal operators of both the smooth quantile functions and nonconvex penalties.
\subsection{Smooth Quantile Losses}\label{sec21}
The quantile loss function, proposed by \cite{Koenker1978Regressions}, is defined as $$\rho_\tau (u) = u(\tau - 1_{u<0})$$ where $\tau \in (0,1)$ and $1_{u<0}$ is an indicator function, which is equal to 1 if $u < 0$, and 0 otherwise. This loss function has been widely studied in linear regression because it can go beyond the typical Gaussian assumptions and effectively resolve heavy-tailed random error distributions.  However,  the traditional quantile loss function may not be suitable for scenarios with noise and outliers, as it aims for precise fitting of a specific section of the data. Studies  conducted by \cite{Aravkin2014Sparse} and \cite{Mkhadri2017A}  have shown the limitations of the traditional quantile loss function and suggested using the smooth quantile loss as a better alternative. 

The two  popular smooth quantile losses proposed by \cite{Jennings1993An} and \cite{Aravkin2014Sparse}  are respectively defined as
\begin{align}\label{c}
\mathcal{L}_{\tau,c}(u) = \begin{cases} 
\tau (u - 0.5c) & \text{if } u \ge c, \\
\frac{\tau u^2}{2c} & \text{if}  \  u \in [0, c) , \\
\frac{\ (1- \tau) u^2}{2c} & \text{if}  \  u \in [-c, 0) , \\
(\tau -1)(u +0.5c)& \text{if}  \  u  <  -c,
\end{cases}
\end{align}
and
\begin{align}\label{ka}
\mathcal{L}_{\tau,\kappa}(u) = \begin{cases} 
\tau (u - \frac{\tau\kappa}{2} )& \text{if } u > \tau\kappa, \\
\frac{u^2}{2\kappa} & \text{if}  \  u \in [(\tau-1)\kappa, \tau\kappa] , \\
(\tau - 1) [u - \frac{(\tau-1)\kappa}{2}]& \text{if}  \  u  <  (\tau-1)\kappa,
\end{cases}
\end{align}
where  $\tau \in (0,1)$,   $c>0$ and $\kappa > 0$ are three  given constants. The smooth quantile loss is termed as a generalized loss because \( \mathcal{L}_{\tau,c}(u) \) and/or \( \mathcal{L}_{\tau,\kappa}(u) \) can closely resemble the quantile loss, least squares loss, asymmetric least squares loss, and Huber loss by adjusting \( \tau \), \( c \)  and \( \kappa \). Obviously, these two smooth quantile losses are the modified versions of the quantile loss function that incorporates smooth functions to replace the non-differentiable portion near the origin. This modification allows for differentiability at the origin, making it easier to analyze theoretically and design algorithms. One manifestation is that \( \mathcal{L}_{\tau,*}(u) \) (\( * \) represents \( c \) or \( \kappa \)) has a Lipschitz continuous first derivative, while quantile loss does not.

Because of these advantages and flexibility,  the smooth quantile losses are widely used in various fields, such as  nonparametric regression (\cite{Hee2011Fast}), functional data clustering (\cite{Kim2020Pseudo}), linear regression (\cite{Zheng2011Gradient}), and penalized linear regression (\cite{Ouhourane2022Group}).
Several algorithms have been proposed to handle smooth quantile regression (with or without penalty terms), including gradient descent (\cite{Zheng2011Gradient}), orthogonal matching pursuit (\cite{Aravkin2014Sparse}) and coordinate descent (\cite{Mkhadri2017A}). It is important to note that the aforementioned algorithms primarily concentrate on convex penalties. Although  \cite{Mkhadri2017A} extended the algorithm to nonconvex cases, its theoretical convergence is not guaranteed. To the best of our knowledge, no ADMM algorithm has been introduced specifically for solving regression models with smooth quantile regression.  The first use of parallel ADMM algorithm to solve regularized smooth quantile regression is also a contribution of this paper.

\subsection{Nonconvex Penalized Smooth Quantile Regression Models}\label{sec22}
Next, we introduce two generalized nonconvex penalties, namely Snet (SCAD + $\ell_2$) in \cite{Wang2010Variable}, and  Mnet (MCP +   $\ell_2$) in \cite{Huang2016THE}, and additionally propose a novel nonconvex penalty, called Cnet (\text{Capped}-$\ell_1$ + $\ell_2$). Capped-$\ell_1$ is a nonconvex regularization proposed by \cite{Zhang2010Analysis} and has garnered attention in \cite{Guan2018An} and \cite{Pan2021Iteratively}. It is defined as $\text{Cap-L1}_{a,\lambda_1}(\bm{\beta}) = \lambda_1 \sum_j \min\{|\beta_j|, a\}$. 
Snet, Mnet and Cnet are defined separately as,
\begin{align*}
\text{Snet}(\bm \beta) &= \text{SCAD}_{a,\lambda_1}(\bm \beta) + \frac{\lambda_2}{2}\|\bm \beta\|_2^2,\\
\text{Mnet}(\bm \beta) &= \text{MCP}_{a,\lambda_1}(\bm \beta) + \frac{\lambda_2}{2}\|\bm \beta\|_2^2, \\
\text{Cnet}(\bm \beta) &= \text{Cap-L1}_{a,\lambda_1}(\bm \beta)  +  \frac{\lambda_2}{2}\|\bm \beta\|_2^2,
\end{align*}
where $a$ is a constant  determined  by the first regularization term, and its suggested value can be found in papers that proposed these  regularization terms.

Combining the smooth quantile loss and the above generalized nonconvex penalty together produces a generalized nonconvex penalized regression model,
\begin{align}\label{NPQR}
\mathop {\arg \min }\limits_{\bm \beta \in \mathbb{R}^p} \quad \sum_{i=1}^{n} \mathcal{L}_{\tau,*} (y_i - \bm x_i^\top \bm \beta)  + P_\lambda(|\bm \beta|).
\end{align}
We abbreviate it as NPSQR.
NPSQR is a special case of regularized M-estimator, whose statistical properties have been studied in \cite{Negahban2010A}, \cite{Li2011Nonconcave}, \cite{Loh2015Regularized}, \cite{Loh2017STATISTICAL} and \cite{Zhou2018A}.

\subsection{Proximal Operator}\label{sec23}
In iterative algorithms, the existence of closed-form solutions for the proximal operator has a significant impact on the algorithm efficiency. 
To facilitate the closed-form solutions for the updates of $\bm \beta$ and $\bm r$ in the ADMM algorithm, we introduce the following proximal operators,
\begin{align}\label{po}
\operatorname{prox}_{\mu, g}(\bm v) = \arg \min_{\bm u} \left\{ g(\bm u) + \frac{\mu}{2}\|\bm u - \bm v\|_2^2\right \},
\end{align}
where $g(\bm u)$ is a  nonnegative real-valued function,  $\bm v$ is a  constant vector, and $\mu > 0$ is a constant. In this paper, the function $g(\bm u)$ is separable or additive, meaning $g(\bm u) = \sum_{j}g(u_j)$. Then, we can  divide the optimization problem (\ref{po}) into multiple  independent univariate problems. Here, we introduce the closed-form solutions for the proximal operators of the aforementioned  losses and penalties, which play a crucial role in our algorithm. The detailed derivation process of these closed-form solutions is provided in Appendix \ref{A}. We are not aware of any prior work that has derived closed-form solutions for these operators (except for some special cases with Snet in \cite{Wang2010Variable} and Mnet in \cite{Huang2016THE}), making our contribution novel in this regard. Furthermore, our derivation method is concise and easily extendable, enabling potential applications beyond the scope of this study.

\textbullet  \quad The closed-form solution for proximal operator of $\mathcal{L}_{\tau,c}(\bm u)$
\begin{align}\label{fc}
\operatorname{prox}_{\mu, \mathcal{L}_{\tau,c}}( v_j)  = \begin{cases}
v_j -  \frac{\tau}{\mu}, & \text{{if }}  v_j \geq c + \frac{\tau}{\mu}, \\
\frac{c\mu v_j }{c\mu + \tau}, & \text{{if }} 0 \le  v_j < c + \frac{\tau}{\mu}, \\
\frac{c\mu v_j }{c\mu + 1 - \tau}, & \text{{if }}  - c + \frac{\tau - 1}{\mu} \le  v_j <  0, \\
 v_j -  \frac{\tau -1}{\mu}, & \text{{if }}  v_j <  - c + \frac{\tau - 1}{\mu}. \\
\end{cases}
\end{align}

\textbullet  \quad The  closed-form solution for proximal operator of $\mathcal{L}_{\tau,\kappa}(\bm u)$
\begin{align}\label{fka}
\operatorname{prox}_{\mu, \mathcal{L}_{\tau,\kappa}}( v_j)  = \begin{cases}
v_j -  \frac{\tau}{\mu}, & \text{{if }}  v_j \geq \tau \kappa + \frac{\tau}{\mu}, \\
\frac{\kappa\mu v_j }{\kappa\mu + 1}, & \text{{if }}  (\tau -1) \kappa + \frac{\tau - 1}{\mu} \le  v_j < \tau \kappa + \frac{\tau}{\mu}, \\
 v_j -  \frac{\tau -1}{\mu}, & \text{{if }}  v_j <   (\tau -1) \kappa + \frac{\tau - 1}{\mu}. \\
\end{cases}
\end{align}
Obviously, when $c=0$ and $\kappa = 0$, the proximal operator of quantile loss in \cite{Yu2017A} and \cite{Gu2018ADMM} is a special case in (\ref{fc}) or (\ref{fka}).

\textbullet \quad  The closed-form solution for proximal operator of Snet  with $(a-1)(\eta + \lambda_2) > 1$:
\begin{align}\label{snet}
\operatorname{prox}_{\eta, P_\lambda}( v_j)  = \begin{cases}
\text{{sign}}( v_j) \cdot \left[ \frac{\eta| v_j| - \lambda_1}{\eta + \lambda_2} \right]_+, & \text{{if }} | v_j| \leq \frac{\lambda_1(1+\eta+\lambda_2)}{\eta}, \\
\text{{sign}}( v_j) \cdot \left[ \frac{(a-1)\eta| v_j| - a \lambda_1}{(a-1) (\eta + \lambda_2) - 1} \right]_+, & \text{{if }} \frac{\lambda_1(1+\eta+\lambda_2)}{\eta} < | v_j| < \frac{a \lambda_1(\eta+\lambda_2)}{\eta}, \\
\frac{\eta  v_j}{\eta + \lambda_2}, & \text{{if }} | v_j| \geq  \frac{a \lambda_1(\eta+\lambda_2)}{\eta}. \\
\end{cases}
\end{align}
Obviously, when $\lambda_2=0$ and $\eta=1$, the closed-form solution for  proximal operator of SCAD in \cite{Fan2001Variable}  is a special case  in (\ref{snet}). Besides, when $\eta = 1$, the solution in (\ref{snet}) is the same as the closed-form solution mentioned in \cite{Wang2010Variable}. Note that $(a-1)(\eta + \lambda_2) > 1$  is also valid in this paper, as $a=3.7$ and $\eta >> 1$. 

\textbullet \quad  The closed-form solution for proximal operator of Mnet:  
\begin{align}\label{mnet}
\operatorname{prox}_{\eta, P_\lambda}( v_j)  = \begin{cases}
\text{{sign}}( v_j) \cdot \left[ \frac{a \eta| v_j| - a \lambda_1}{a(\eta + \lambda_2) -1} \right]_+, & \text{{if }} | v_j| <  \frac{a \lambda_1(\eta+\lambda_2)}{\eta}, \\
\frac{\eta  v_j}{\eta + \lambda_2}, & \text{{if }} | v_j| \geq  \frac{a \lambda_1(\eta+\lambda_2)}{\eta}. \\
\end{cases}
\end{align}
Obviously, when $\lambda_2=0$ and $\eta=1$, the proximal operator of MCP in \cite{Zhang2010Nearly}  is a special case  in (\ref{mnet}). In addition, when $\eta = 1$, the solution in (\ref{mnet}) is the same as the solution mentioned in \cite{Huang2016THE}.

\textbullet \quad  The closed-form solution for proximal operator of Cnet: 
\begin{align}
\operatorname{prox}_{\eta, P_\lambda}( v_j)  = \begin{cases}
\text{{sign}}( v_j) \cdot \left[ \frac{ \eta| v_j| -  \lambda_1}{\eta + \lambda_2} \right]_+, & \text{{if }} | v_j| <  \frac{a (\eta+\lambda_2)}{\eta}, \\
\frac{\eta  v_j}{\eta + \lambda_2}, & \text{{if }} | v_j| \geq  \frac{a (\eta+\lambda_2)}{\eta} .\\
\end{cases}
\end{align}

In addition to these two smooth quantile losses and three nonconvex combined penalties, our PIPADMM package also offers closed-form solutions  for proximal operators of  least squares loss and its asymmetric variant,  quantile regression loss,  Huber loss, and elastic-net regularization (\cite{Zou2005Regularization}). Naturally, our algorithm (R package) can also be applied to solve these regression problems with elastic-net. Some of these proximal operators  have already obtained the closed-form solutions from previous  research in \cite{Liang2024Linearized} and its cited literature, and the closed-form solutions of all the proximal operators mentioned above can be derived using the similar methods as presented in this paper. As a result, we will not delve into detailed introductions for them.

\section{The Parallel LADMM Algorithm}\label{sec3}
The alternating direction method of multipliers (ADMM) is an algorithm employed to solve problems characterized by a two-block separable objective function with equality constraints. By setting $\bm r = \bm y - \bm X \bm \beta$, the given NPSQR problems in (\ref{NPQR}) can be transformed into the following form,
\begin{align}\label{CNPSQR}
\min_{\bm \beta, \bm r} & \quad \mathcal{L}_{\tau,*}(\bm r)  + {P}_\lambda(|\bm \beta|), \notag \\
\text{s.t.} & \quad \bm X \bm \beta + \bm r = \bm y,
\end{align}
where $\mathcal{L}_{\tau,*}(\bm r) =   \sum_{i=1}^{n} \mathcal{L}_{\tau,*}(r_i)$.
To solve NPSQR using ADMM, the augmented Lagrangian of  (\ref{CNPSQR}) needs to be formulated as,
\begin{equation}\label{AL}
{L_\mu }(\bm \beta, \bm r, \bm d) = \mathcal{L}_{\tau,*}(\bm r) +  {P}_\lambda(|\bm \beta|) - {\bm d^\top}(\bm{X\beta} + \bm{r} - \bm y) + \frac{\mu }{2}\left\|\bm{X\beta} + \bm{r} - \bm y \right\|_2^2.
\end{equation}
Here, $\bm d \in \mathbb R^n$ represents the dual variable, and $\mu > 0$ is a tunable augmentation parameter.
The ADMM algorithm alternates between the following iterative steps with given $\bm r^0$ and $\bm d^0$,
\begin{equation}\label{twoupadmm}
\left\{ \begin{array}{l}
\bm \beta^{k+1} \ =\mathop {\arg \min }\limits_{\bm \beta} \left\{  {P}_\lambda(|\bm \beta|) + \frac{\mu }{2}\left\|\bm{X\beta} + \bm{r}^k - \bm y - \bm d^k/\mu \right\|_2^2     \right\},\\
\bm r^{k+1} \ =\mathop {\arg \min }\limits_{\bm r} \left\{ \mathcal{L}_{\tau,*}(\bm r) +  \frac{\mu }{2}\left\|\bm{X}\bm{\beta}^{k+1} + \bm{r} - \bm y - \bm d^k/\mu \right\|_2^2 \right\},\\
\bm{d}^{k+1}=\bm{d}^{k}-\mu(\bm{X}\bm{\beta}^{k+1} + \bm{r}^{k+1} - \bm y).
\end{array} \right.
\end{equation}
Similar iterative  schemes have been used by \cite{Yu2017A} and \cite{Gu2018ADMM} to solve the convex penalized quantile regression. However, the  $\bm \beta$-update in (\ref{twoupadmm}) is a regression problem similar to penalized least squares regression, which does not have a closed-form solution. This update is computationally expensive because it necessitates numerical optimization techniques like the coordinate descent algorithm, leading to a double-loop algorithm. Next, we will introduce the LADMM algorithm to address this challenge.
\subsection{LADMM Algorithm}\label{sec31}
Linearized ADMM (LADMM)  algorithms have found wide application in convex sparse statistical learning models, including the Dantzig selector in \cite{Wang2012The}, sparse group and fused Lasso methods in \cite{Li2014Linearized}, sparse quantile regression  in \cite{Gu2018ADMM}, and sparse support vector machines  in \cite{Liang2024Linearized}.
The fundamental idea of the LADMM algorithm involves linearizing the quadratic term using a nontrivial matrix, thereby creating a proximal operator. We solve (\ref{CNPSQR}) using LADMM through the following iteration,
\begin{equation}\label{twoupladmm1}
\left\{ \begin{array}{l}
\bm \beta^{k+1} \ =\mathop {\arg \min }\limits_{\bm \beta} \left\{  {P}_\lambda(|\bm \beta|) + \frac{\mu }{2}\left\|\bm{X\beta} + \bm{r}^k - \bm y - \bm d^k/\mu \right\|_2^2 +   \frac{1}{2}\left \| \bm \beta - \bm \beta^k \right \|_{\bm S}^2    \right\},\\
\bm r^{k+1} \ =\mathop {\arg \min }\limits_{\bm r} \left\{ \mathcal{L}_{\tau,*}(\bm r) +  \frac{\mu }{2}\left\|\bm{X}\bm{\beta}^{k+1} + \bm{r} - \bm y - \bm d^k/\mu \right\|_2^2 \right\},\\
\bm{d}^{k+1}=\bm{d}^{k}-\mu(\bm{X}\bm{\beta}^{k+1} + \bm{r}^{k+1} - \bm y),
\end{array} \right.
\end{equation}
where  $\bm S = \eta  \bm I_p - \mu \bm X^\top \bm X$ and $\left \| \bm \beta - \bm \beta^k \right \|_{\bm S}^2 =   (\bm \beta - \bm \beta^k)^\top \bm S (\bm \beta - \bm \beta^k)$. To ensure the convergence of the algorithm, we need $\bm S$ to be a positive-definite matrix, i.e., $\eta > \text{eigen}(\mu \bm X^\top \bm X)$, where $\text{eigen}(\mu \bm X^\top \bm X)$ denotes the maximum eigenvalue of  $\mu \bm X^\top \bm X$.
\begin{algorithm}\small
\caption{\small{LADMM for solving NPSQR.}}
\label{alg1}
\begin{algorithmic}
\STATE {\textbf{Pre-computation}:  $\eta =\text{eigen}(\mu \bm X^\top \bm X )$.} 
\STATE {\textbf{Input:} Observation data: $\boldsymbol{X},\boldsymbol{y}$;  initial primal variables: ${\boldsymbol{\beta}}^0,\boldsymbol{r}^0$; initial dual variables $\boldsymbol{d}^0$; augmented parameters: $\mu$; penalty parameter: $\lambda_1$, $\lambda_2$;  selected parameters: $\tau \in (0,1)$,  $c>0$ or
$\kappa>0$.}
\STATE {\textbf{Output:}  The last iteration solution $\boldsymbol{\beta}^K$. }
\STATE {\textbf{while} not converged \textbf{do}}
\STATE {\quad 1. Update $\bm \beta^{k+1}_j  =  \operatorname{prox}_{\eta, P_\lambda} \left( \bm \beta_j^k - \frac{ \left[ \mu \bm X^\top(\bm{X\beta}^k + \bm{r}^k - \bm y - \bm d^k/\mu)\right]_j }{\eta}\right), j=1,2,\dots,p,$}
\STATE {\quad 2. Update $\bm r_i^{k+1}  =  \operatorname{prox}_{\mu, \mathcal{L}_{\tau,*}}\left([\bm y + \bm d^k/\mu - \bm{X}\bm{\beta}^{k+1}]_i  \right), i=1,2,\dots,n,$}
\STATE {\quad 3. Update $\bm{d}^{k+1}=\bm{d}^{k}-\mu(\bm{X}\bm{\beta}^{k+1} + \bm{r}^{k+1} - \bm y).$}
\STATE {\textbf{end while}}
\STATE {\textbf{return} solution}
\end{algorithmic}
\label{alg1}
\end{algorithm}

For the $\bm \beta$-update,  omitting constant terms unrelated to $\bm \beta$,  the subproblem  becomes
\begin{align}\label{npbeta}
\bm \beta^{k+1} \ =\mathop {\arg \min }\limits_{\bm \beta} \left\{ {P}_\lambda(|\bm \beta|) + \frac{\eta }{2}\left\| \bm \beta - \bm \beta^k + \frac{\mu \bm X^\top(\bm{X\beta}^k + \bm{r}^k - \bm y - \bm d^k/\mu)}{\eta} \right\|_2^2  \right\}.
\end{align}
Obviously, both $\bm \beta^{k+1}$ and $\bm r^{k+1}$ can be represented as proximal operators, that is,
\begin{align}\label{uadmm}
\left\{
    \begin{aligned}
        &\bm \beta^{k+1}  =  \operatorname{prox}_{\eta, P_\lambda} \left( \bm \beta^k - \frac{\mu \bm X^\top(\bm{X\beta}^k + \bm{r}^k - \bm y - \bm d^k/\mu)}{\eta}\right), \\
        &\bm r^{k+1}  =  \operatorname{prox}_{\mu, \mathcal{L}_{\tau,*}}\left(\bm y + \bm d^k/\mu - \bm{X}\bm{\beta}^{k+1}  \right).
    \end{aligned}
\right.
\end{align}
To summarize, the iterative scheme of LADMM for (\ref{CNPSQR}) is described in Algorithm \ref{alg1} .

\subsection{Parallel LADMM Algorithm}\label{sec32}
This subsection delves into the implementation of parallel LADMM algorithms designed for handling large-scale data. The main objective is to address the issue in a distributed manner, with each processor tasked with handling a subset of the training data. This distributed approach offers significant advantages when facing with a high volume of training examples that cannot be practically processed on a single machine, or when the data is inherently distributed or stored in a similar manner.
\begin{figure}[htbp]
\centering
\caption{The parallel LADMM algorithm}\label{fig1}
\includegraphics[width=15cm,height=9cm]{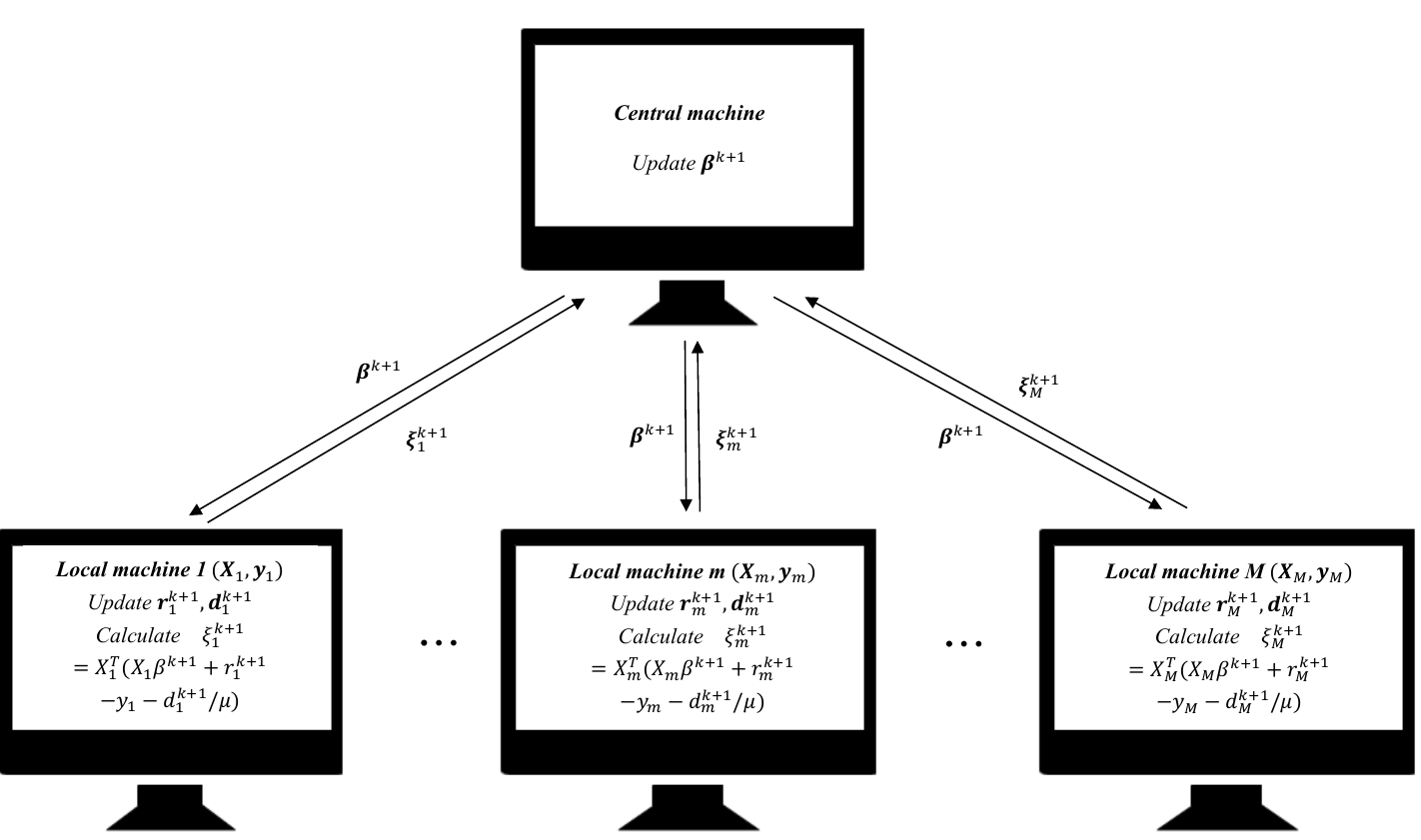}
\end{figure}

We divide $\bm X$ and $\bm y$ into $M$ row blocks using the following approach:
\begin{align}\label{Xy}
\bm y = (\bm y_1^\top, \bm y_2^\top, \dots, \bm y_M^\top)^\top, \quad \bm X =(\bm X_1^\top, \bm X_2^\top, \dots, \bm X_M^\top)^\top,
\end{align}
where $\bm{X}_m \in \mathbb{R}^{n_m \times p}$ and $\bm{y}_m \in \mathbb{R}^{n_m}$, with $\sum_{m=1}^{M} n_{m} = n$. Here, $\bm X_m$ and $\bm y_m$ denote the $m$-th block of data, which will be processed by the $m$-th local machine. Correspondingly, $\bm r$ and $\bm d$ can also be divided into $M$ blocks:
\begin{align}\label{rb}
\bm r = (\bm r_1^\top, \bm r_2^\top, \dots, \bm r_M^\top)^\top, \quad \bm d =(\bm d_1^\top, \bm d_2^\top, \dots, \bm d_M^\top)^\top.
\end{align}
Then, we can rewrite problem (\ref{CNPSQR}) into the following equivalent problem:
\begin{align}\label{MCNPSQR}
\min_{\bm \beta,\{ \bm r_m \}_{m=1}^{M}}& \quad  \sum_{m=1}^{M} \sum_{l=1}^{n_m} \mathcal{L}_{\tau,*}( r_{m,l}) + {P}_\lambda(|\bm \beta|), \notag \\
\text{s.t.} & \quad \bm X_m \bm \beta + \bm r_m = \bm y_m, \quad m=1,2,\dots,M,
\end{align}
where $\bm r_m = (r_{m,1},r_{m,2},\dots,r_{m,n_m})^\top$.
Similar to (\ref{AL}), the augmented Lagrangian for problem (\ref{MCNPSQR}) is given by
\begin{align}\label{AL2}\small
{L_\mu }(\bm \beta, \{\bm r_m\}_{m=1}^{M}, \{\bm d_m\}_{m=1}^{M}) & = \sum_{m=1}^{M} \sum_{l=1}^{n_m} \mathcal{L}_{\tau,*}( r_{m,l}) +  {P}_\lambda(|\bm \beta|) - \sum_{m=1}^{M}{\bm d_m^\top}(\bm{X_m\beta} + \bm{r}_m - \bm y_m)  \notag \\
& + \frac{\mu }{2}\sum_{m=1}^{M} \left\|\bm{X_m\beta} + \bm{r}_m - \bm y_m \right\|_2^2,
\end{align} 
where $\bm d_m \in \mathbb R^{n_m}$ is the dual variable. Then,  we can solve (\ref{MCNPSQR}) using parallel LADMM through the following iteration,
\begin{equation}\label{twoupladmm2}
\left\{ \begin{array}{l}
\bm \beta^{k+1} \ =\mathop {\arg \min }\limits_{\bm \beta} \left\{  {P}_\lambda(|\bm \beta|) + \frac{\mu }{2}  \sum_{m=1}^{M} \left\|\bm{X}_m\bm{\beta} + \bm{r}_m^k - \bm y_m - \bm d_m^k/\mu \right\|_2^2 +   \frac{1}{2}\left \| \bm \beta - \bm \beta^k \right \|_{\bm S}^2    \right\},\\
\bm r_m^{k+1} \ =\mathop {\arg \min }\limits_{\bm r_m} \left\{ \sum_{l=1}^{n_m} \mathcal{L}_{\tau,*}( r_{m,l})+  \frac{\mu }{2}\left\|\bm{X}_m\bm{\beta}^{k+1} + \bm{r}_m - \bm y_m - \bm d_m^k/\mu \right\|_2^2 \right\}, m=1,2,\dots,M,\\
\bm{d}_m^{k+1}=\bm{d}_m^{k}-\mu(\bm{X}_m\bm{\beta}^{k+1} + \bm{r}_m^{k+1} - \bm y_m), m=1,2,\dots,M,
\end{array} \right.
\end{equation}
where  $\bm S = \eta  \bm I_p - \mu  \sum_{m=1}^{M}  \bm X_m^\top \bm X_m = \eta  \bm I_p - \mu    \bm X^\top \bm X$. The steps that need to execute parallel processes in iteration are $\bm r_m$ and $\bm d_m$.   However, the central machine updating $\bm \beta$  encounters two issues due to not loading data. The first issue is that the  update of $\bm \beta$ requires loading the collected data $\bm X, \bm y$. The second one is that the calculation of $\eta$ also depends on the collected data.

The above two issues seem to make this  parallel   LADMM algorithm impossible to implement. Fortunately, continuing to derive the specific iteration steps for $\bm \beta$, we found that both of these issues can be solved. Similar to (\ref{uadmm}),  the iteration of $\bm \beta$ and $\bm r_m$ can be represented as 
\begin{align}\label{Mupdate}
\left\{
    \begin{aligned}
        &\bm \beta^{k+1}  =  \operatorname{prox}_{\eta, P_\lambda} \left( \bm \beta^k - \frac{\mu \sum_{m=1}^{M} \bm X_m^\top(\bm{X}_m \bm \beta^k + \bm{r}_m^k - \bm y_m - \bm d_m^k/\mu)}{\eta}\right), \\
        &\bm r_m^{k+1}  =  \operatorname{prox}_{\mu, \mathcal{L}_{\tau,*}}\left(\bm y_m + \bm d_m^k/\mu - \bm{X}_m\bm{\beta}^{k+1}  \right), m=1,2,\dots,M.
    \end{aligned}
\right.
\end{align}
Let   $\bm{\xi}_m^{k}=\bm X_m^\top(\bm{X}_m \bm \beta^{k} + \bm{r}_m^{k} - \bm y_m - \bm d_m^{k}/\mu)$, $m=1,2,\dots,M$. From the iteration sequence, it can be seen that in the $k+1$ iteration, $\bm \xi_m^{k}$ can be generated by each local machine and transmitted to the central machine.  Then,  
$$\bm \beta^{k+1}  =  \operatorname{prox}_{\eta, P_\lambda} \left(\bm \beta^k - \frac{\mu }{\eta} \sum_{m=1}^{M} \bm \xi_m^{k} \right).$$  Therefore, the first issue caused by the central machine not loading data has been resolved. 

The second issue is the calculation of $\eta$. Each local machine only loads a portion of the data, which means we can only obtain $\eta_m = \text{eigen}(\mu \bm X_m^\top \bm X_m)$ on each local machine.  Note that $\sum_{m=1}^{M}\eta_m \ge \text{eigen} (\mu \bm X^\top \bm X)$. Then, in the specific implementation of the parallel algorithm, we can choose $\sum_{m=1}^{M}\eta_m$ as the value for $\eta$. In practical operation, we only need to calculate $\eta_m$ once at the beginning of the iteration and then pass it on to the central machine. Therefore, both issues related to the implementation of parallel LADMM have been resolved. 

To sum up, the iterative scheme of parallel LADMM for (\ref{CNPSQR}) can be described in Algorithm \ref{alg2} and  visualized Figure \ref{fig1}.
\begin{algorithm}\small
\caption{\small{Parallel LADMM for solving NPSQR}}
\label{alg2}
\begin{algorithmic}
\STATE {\textbf{Pre-computation}:  $\eta_m =\text{eigen}(\mu \bm X_m^\top \bm X_m )$, $m= 1,2,\dots,M$.} 
\STATE {\textbf{Input:} $\bullet$ Central machine: augmented parameter $\mu$; tuning parameters $\lambda_1$ \text{and} $\lambda_2$;  $\eta = \sum_{m=1}^{M}\eta_m$; and initial iteration variable $\bm \beta^0$.\\
\qquad \  \ \ \ \ $\bullet$ The $m$-th local machine:  observation data $\{\boldsymbol{X}_m,\boldsymbol{y}_m\}_{m=1}^{M}$; augmented parameter $\mu$; selected parameters $\tau \in (0,1)$, and $c>0$ or $\kappa>0$; and initial iteration variables $\bm \beta^0$,  $\bm r_m^0$, $\bm d_m^0$ and $\bm \xi_m^0 =\bm X_m^\top(\bm{X}_m \bm \beta^{0} + \bm{r}_m^{0} - \bm y_m - \bm d_m^{0}/\mu)$.}
\STATE {\textbf{Output:}  The last iteration solution $\boldsymbol{\beta}^K$. }
\STATE {\textbf{while} not converged \textbf{do}}
\STATE {\ \textbf{Central machine}: 1. Receive $\eta_m$ (only once) and $\bm \xi_m^k$ transmitted by $M$ local machines,\\ \qquad \qquad \qquad  \qquad \quad \ 2.  Update $\bm \beta^{k+1}_j  =  \operatorname{prox}_{\eta, P_\lambda} \left( \bm \beta_j^k - \frac{\left( \mu \sum_{m=1}^{M} \bm \xi_m^k \right)_j}{\eta}\right), j=1,2,\dots,p,$ \\
\qquad \qquad \qquad  \qquad \quad \ 3. Send $\bm \beta^{k+1}$ to the local machines.
}
\STATE {\ \textbf{Local machines}: \ \  for $m =1 ,2, \dots, M$ (in parallel) \\
\qquad \qquad \qquad  \qquad \quad \ 1. Receive $\bm \beta^{k+1}$  transmitted by the central machine, \\
\qquad \qquad \qquad  \qquad \quad \ 2.  Update $\bm r_{m,l}^{k+1}  =  \operatorname{prox}_{\mu, \mathcal{L}_{\tau,*}}\left([\bm y_m + \bm d_m^k/\mu - \bm{X}_m\bm{\beta}^{k+1}]_l  \right), m=1,2,\dots,M$,  \\
\qquad \qquad \qquad  \qquad \quad \quad \ \ $l=1,2,\dots,n_m$, \\
\qquad \qquad \qquad  \qquad \quad \ 3. Update $\bm{d}_m^{k+1}=\bm{d}_m^{k}-\mu(\bm{X}_m\bm{\beta}^{k+1} + \bm{r}_m^{k+1} - \bm y_m), m=1,2,\dots,M,$ \\
\qquad \qquad \qquad  \qquad \quad \
4. Calculate $\bm{\xi}_m^{k+1}=\bm X_m^\top(\bm{X}_m \bm \beta^{k+1} + \bm{r}_m^{k+1} - \bm y_m - \bm d_m^{k+1}/\mu), m=1,2,\dots,M,$ \\
\qquad \qquad \qquad  \qquad \quad \ 5. Send $\bm \xi^{k+1}_m$  to the central machine.
}
\STATE {\textbf{end while}}
\STATE {\textbf{return} solution}.
\end{algorithmic}
\end{algorithm}

Next, we will theoretically describe the relationship between Algorithm \ref{alg1} and Algorithm  \ref{alg2}. In fact, when $M=1$, the two algorithms are the same, so the following discussion focuses on the case of $M \ge 2$ in Algorithm \ref{alg2}. For ease of distinction,  let us denote  $\left\{ \hat{\bm \beta}^k, \hat{\bm r}^k, \hat{\bm d}^k\right\}$ as the $k$-th iteration results of Algorithm  \ref{alg1}, and  $\left\{ \tilde{\bm \beta}^k, \tilde{\bm r}^k , \tilde{\bm d}^k \right\}$ as the $k$-th iteration results of Algorithm  \ref{alg2}.
Then, we can draw the following  surprising conclusion.
\begin{thm}\label{TH1}
If we use the same initial  iteration variables ($\{ \hat{\bm \beta}^0, \hat{\bm r}^0, \hat{\bm d}^0 \} = \{ \tilde{\bm \beta}^0, \tilde{\bm r}^0 , \tilde{\bm d}^0 \}$) and $\eta$,  the iterative  solutions  obtained  by these two algorithms are actually the same, i.e.,
\begin{align}
\left\{ \hat{\bm \beta}^k, \hat{\bm r}^k, \hat{\bm d}^k\right\} = \left\{ \tilde{\bm \beta}^k, \tilde{\bm r}^k , \tilde{\bm d}^k \right\}, \ \text{for} \ \text{all} \  {k}.
\end{align}
\end{thm}

This theorem states that as long as Algorithm \ref{alg1} and Algorithm \ref{alg2} ($M \ge 2$) adopt the same  initial value and $\eta$, the iterative solution remains the same regardless of how the samples are divided.
In other words, as long as $\eta$ is the same, changes in $M$ and the number of samples loaded by each machine will not affect the iterative solution of parallel LADMM. While this conclusion may not seem immediately intuitive, the proof for it is remarkably simple and is given in Appendix \ref{B1}. 

However, as discussed above in solving the calculation problem of $\eta$ in parallel algorithms, as the number of local machines increases, the value of $\eta$ will also increase.
There have been some studies on the impact of the  $\eta$ size in the linearized ADMM
algorithm. Their conclusion was that a large $\eta$ value causes the algorithm to converge slower, and there may be no significant changes in the convergence solution of the iteration.  Interested readers can refer to the research conducted by \cite{He2020Optimally} and its references. Next, we will discuss the convergence of the algorithm.

Let us assume that $\bm X^\top \bm X > \underline{\mu} \bm I_p $, where $\underline{\mu}>0$ is a constant. It is worth noting that this assumption is, in fact, the lower restricted eigenvalue condition that is required in many penalized linear models as mentioned in \cite{Wang2020A} and the references cited therein. We now demonstrate the convergence of the LADMM algorithm, and the proof is given in Appendix \ref{B2}.
\begin{thm}\label{TH2}
Let the sequence $\boldsymbol{w}^{k}=\{\boldsymbol{\beta}^k, \boldsymbol{r}^k, \boldsymbol{d}^k\}$ be generated by Algorithms \ref{alg1} or  \ref{alg2}   with an arbitrary initial feasible solution $\boldsymbol{w}^{0}=\{\boldsymbol{\beta}^0, \boldsymbol{r}^0, \boldsymbol{d}^0\}$.  Then, with the condtions $\mu >  \frac{\sqrt{2n}}{\min\{c,\kappa \}}$ and $\eta > \text{eigen}(\mu {{\boldsymbol{X}}}^\top {\boldsymbol{X}})$,  the sequence $\boldsymbol{w}^{k}$ converges to $\boldsymbol{w}^{*}=\{ \bm{\beta}^*, \boldsymbol{r}^*, \boldsymbol{d}^*\}$, where  $\boldsymbol{w}^{*}$ is a critical point of $L_\mu$. 
\end{thm}

\begin{rem}
To solve the NPSQR with $\mathcal{L}_{\tau,c}$, the condition $\mu > \frac{\sqrt {2n} \max\{\tau, 1-\tau\}}{c}$ guarantees convergence of the algorithm. On the other hand, for the NPSQR with $\mathcal{L}_{\tau,\kappa}$, the algorithm can converge as long as $\mu > \frac{\sqrt {2n}}{\kappa}$. The reason behind the different requirements for $\mu$ in these two models lies in the disparity of the Lipschitz constant of the first derivative of the two smooth quantile losses. 
\end{rem}

A reasonable concern arises regarding the assumption about $\mu$ in this theorem. When $n$ is large and either $c$ or $k$ takes small values, $\mu$ may become large. However, this does not affect the convergence of the algorithm. 
An intuitive explanation lies in the optimization formula (\ref{AL}). 
Our loss function is not divided by $n$ like the empirical loss. 
Compared to the case of empirical loss, our objective function is multiplied by a factor of $n$.
 If we consider the normalized loss function $\frac{\sum_{m=1}^{M} \sum_{l=1}^{n_m} \mathcal{L}_{\tau,*}( r_{m,l})}{n}$ instead of the unnormalized one $\sum_{m=1}^{M} \sum_{l=1}^{n_m} \mathcal{L}_{\tau,*}( r_{m,l})$, then $\mu$ only needs to be greater than $\frac{\sqrt{2}}{\sqrt{n} \min\{c,\kappa \}}$. 
When $n$ is relatively large, $\mu$ may be relatively small or even close to $0$.

\subsection{Comparison with Other Parallel ADMM Algorithms}\label{sec33}
There is a dearth of parallel algorithms for addressing the NPSQR issue. However, recent research has introduced parallel ADMM algorithms for nonconvex penalized quantile regression (NPQR), such as \cite{Yu2017A}, \cite{Fan2021Penalized} and \cite{Wen2023Feature}. 
Ignoring some unrelated constant terms,  as the values of $c$ and $\kappa$ approach 0, the function $\mathcal{L}_{\tau,*}$ converges to the quantile loss. Consequently, our algorithm can be applied to solve the NPQR problem.
To provide a more comprehensive comparison, we review several parallel ADMM algorithms that address the NPQR problem.

\textbullet  \quad QPADM in  \cite{Yu2017A}. 
To have a parallel structure, auxiliary variables $\{\bm \beta_m = \bm \beta\}_{m=1}^{M}$ are added to (\ref{CNPSQR}). This modification transforms the constrained optimization problem into the following form,
\begin{align}\label{CCNPSQR}
\min_{\bm \beta, \{\bm r_m, \bm \beta_m\}_{m=1}^M}  \quad &\sum_{m=1}^{M} \rho_{\tau}(\bm r_m) + {P}_\lambda(|\bm \beta|), \notag \\
\text{s.t.}  \quad  \bm X_m \bm \beta_m +  \bm r_m &= \bm y_m, \bm \beta_m = \bm \beta, m=1,2,\dots,M.
\end{align}
Problem (\ref{CCNPSQR}) has the following augmented Lagrangian,
\begin{align}\label{SAL}
{L_\mu }(\bm \beta, \{\bm r_m, \bm \beta_m, \bm d_m\}_{m=1}^M) =\sum_{m=1}^{M} \rho_{\tau}(\bm r_m) +  {P}_\lambda(|\bm \beta|) - \sum_{m=1}^{M} {\bm d_{1m}^\top}(\bm{X}_m \bm{\beta}_m + \bm{r}_m - \bm y_m) \notag \\
 + \frac{\mu }{2} \sum_{m=1}^{M} \left\|\bm{X_m} \bm{\beta}_m + \bm{r}_m - \bm y_m \right\|_2^2 - \sum_{m=1}^{M} {\bm d_{2m}^\top}(\bm{\beta}_m - \bm{\beta}) +  \frac{\mu }{2}  \sum_{m=1}^{M}  \left\| \bm{\beta}_m - \bm{\beta} \right\|_2^2,
\end{align} 
where $\bm d_{1m} \in \mathbb R^{n_m}$ and $\bm d_{2m} \in \mathbb R^p$ are the dual variables. The parallel ADMM iterative  scheme of (\ref{SAL}) is 
\begin{equation}\label{threeupadmm}\small
\left\{ \begin{array}{l}
\bm \beta^{k+1} \ = \mathop {\arg \min }\limits_{\bm \beta} \left\{   {P}_\lambda(|\bm \beta|) + \frac{\mu }{2} \sum_{m=1}^{M} \left\|\bm{\beta}_m^k - \bm \beta - \bm d_{2m}^k/\mu \right\|_2^2    \right\},\\
\bm r_m^{k+1} \ =\mathop {\arg \min }\limits_{\bm r_m} \left\{ \rho_{\tau}(\bm r_m) +  \frac{\mu }{2}\left\|\bm{X}_m \bm{\beta}_m^{k+1} + \bm{r}_m - \bm y_m - \bm d_{1m}^k/\mu \right\|_2^2 \right\}, m=1,2,\dots,M,\\
\bm \beta_m^{k+1} \ =\mathop {\arg \min }\limits_{\bm \beta_m} \left\{\left\|\bm{X}_m \bm{\beta}_m + \bm{r}_m^k - \bm y_m - \bm d_{1m}^k/\mu \right\|_2^2 + \left\|\bm{\beta}_m - \bm \beta^{k+1} - \bm d_2^k/\mu \right\|_2^2     \right\},m=1,2,\dots,M,\\
\bm{d}_{1m}^{k+1}=\bm{d}_{1m}^{k}-\mu(\bm{X}_m\bm{\beta}_m^{k+1} + \bm{r}_m^{k+1} - \bm y_m),m=1,2,\dots,M,\\
\bm{d}_{2m}^{k+1}=\bm{d}_{2m}^{k}-\mu(\bm{\beta}_m^{k+1} - \bm \beta^{k+1}), m=1,2,\dots,M,.
\end{array} \right.
\end{equation}
Note that when $M=1$, the algorithm becomes a nonparallel version of ADMM. 
Increasing the consensus constraint $\bm \beta_m = \bm \beta$ aids in eliminating internal loops and adapting to parallel frameworks.  However, this change introduces an additional step in solving $\bm \beta_m$, which involves finding the inverse of $\bm X_m^\top \bm X_m + \bm I_{n_m}$. Although \cite{Yu2017A} suggested using the Woodbury matrix identity to alleviate the computational burden of inverting these matrices, matrix multiplication can still be time-consuming, especially when $p$ is large. 
Motivated by the maximization in QICD proposed by \cite{Peng2015An}, QPADM transforms the nonconvex problem into weighted $\ell_1$ soft-shrinkage operator for $\bm \beta^{k+1}$-subproblem. However, this approximation may result in an excessive number of iterative steps.

\textbullet  \quad QPADMslack in \cite{Fan2021Penalized}. 
Inspired by \cite{Guan2018An}, \cite{Fan2021Penalized} relaxed $\bm r_m$ in (\ref{SAL}) to $\bm u_m- \bm v_m$ with $\bm u_m \ge \bm 0$ and $\bm v_m \ge \bm 0$ ($m=1,2,\dots,M$), and made an improvement in solving the $\bm \beta^{k+1}$-subproblem.  These changes can reduce the iteration steps and improve the computational accuracy of QPADM in calculating NPQR. The constrained optimization problem is 
\begin{align}\label{QPs}
\min_{\bm \beta, \{\bm u_m \ge \bm 0, \bm v_m \ge \bm 0, \bm \beta_m\}_{m=1}^M}  \quad &\sum_{m=1}^{M}  \left({\tau}\bm u_m +  (1-\tau) \bm v_m \right)+ {P}_\lambda(|\bm \beta|), \notag \\
\text{s.t.}  \quad &  \bm X_m \bm \beta_m +  \bm u_m - \bm v_m = \bm y_m,  \notag \\
& \bm \beta_m = \bm \beta, m=1,2,\dots,M.
\end{align}
The iteration of QPADMslack shares similar steps with that of QPADM, yet two main differences exist. Firstly, the $\bm r$ subproblem is replaced by the $\bm u$ and $\bm v$ subproblems. Secondly, in contrast to the approximate solution of $\bm \beta$ in QPADM, QPADMslack utilizes the approximate closed-form solution for SCAD and MCP, as proposed by \cite{Gong2013A} and \cite{Guan2018An}.
The term ``approximate closed-form solution" means finding the minimum point by comparing the minimum value of the objective function among several candidate minimum points. By contrast, our LADMM algorithm utilizes the derived proximal operators in Section \ref{sec23} to directly provide a closed-form solution for this nonconvex problem.
\begin{figure}[H]
\centering
\caption{\footnotesize{Average time in 100 repetitions for computing  the maximum eigenvalue of a  matrix.}}\label{fig2}
\includegraphics[width=12cm]{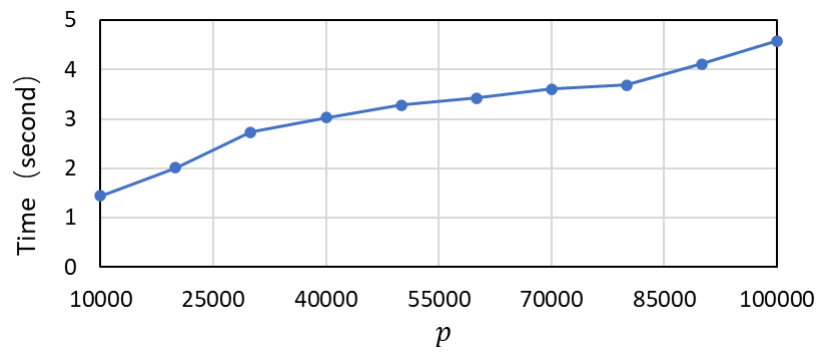}
\end{figure}

The above two algorithms are parallel algorithms specifically designed for working with sample data. The QPADM algorithm necessitates iteration over $(2M+1)p+2n$ variables, whereas the QPADMslack algorithm requires iteration over $(2M+1)p+3n$ variables. When $p$ is very large,  the Woodbury matrix identity that needs to be used is
\begin{align*}
(\bm X_m^\top \bm X_m + \bm I_p)^{-1}  = \bm I_p - \bm X_m^\top(\bm X_m \bm X_m^\top + \bm I_{n_m})^{-1}\bm X_m.
\end{align*}
The matrix multiplication here is quite time-consuming. Referring to Algorithm \ref{alg2}, we observe that our LADMM algorithm only requires iterating over $p+2n$ variables and computing the maximum eigenvalue of the matrix $\mu \bm X^\top \bm X$ (or $\{\mu \bm X_m^\top \bm X_m\}_{m=1}^M$). According to the suggestion by  \cite{Liang2024Linearized}, we use the power method provided by \cite{Golub1996Matrix}  to compute the maximum eigenvalue in the  LADMM algorithm. This method is very efficient. We illustrate the time required to calculate the maximum eigenvalue as $p$ changes in Figure \ref{fig2}. The R code for computing the maximum eigenvalue is also available in our  R package.

\textbullet  \quad  Feature-splitting ADMM  in \cite{Wen2023Feature}. Inspired by  the three-block ADMM in \cite{Sun2015A},   \cite{Wen2023Feature} proposed a feature-splitting algorithm for solving  $\ell_1$ quantile regression.   Furthermore, guided by the theoretical results of \cite{Fan2014Strong}, they utilized the LLA algorithm to solve NPQR. In particular, they used $\ell_1$ quantile regression as the initial value and iterated the calculation several  times to obtain the solution, with a high probability of convergence to the true solution of NPQR.   Feature-splitting ADMM (FSADMM)  requires iterating over $p+(3M-1)n$ variables. Therefore, when $n$  is relatively large, FSADMM may encounter computational difficulties with many local machines.

\section{Simulation Studies}\label{sec4}
In this section, we aim to demonstrate the model selection, estimation accuracy, and computational efficiency of the proposed LADMM algorithm in both nonparallel and parallel environments. To accomplish this, we apply the LADMM algorithm to solve various problems, including NPQR (\cite{Wang2012Quantile}), nonconvex penalized least squares and Huber regression (\cite{Fan2018I-LAMM} and \cite{Pan2021Iteratively}), as well as NPSQR (\cite{Mkhadri2017A}).
To choose the optimal values for the regularization parameters $\lambda_1$ and/or $\lambda_2$, we follow the approach proposed by \cite{Lee2014Model} and \cite{Yu2017A}. We minimize the HBIC criterion, defined as
\begin{align}
\text{HBIC}(\lambda_1, \lambda_2) = \log\left(\sum_{i=1}^{n} \mathcal{L}(y_i - \bm x_i^\top \hat{\bm \beta}_{\lambda_1, \lambda_2})\right) + |S_{\lambda_1, \lambda_2}| \frac{\log(\log n)}{n} C_n.
\end{align}
Here, $\mathcal{L}$ represents a specific loss function, and $\hat{\bm \beta}_{\lambda_1, \lambda_2}$ corresponds to the nonconvex estimator obtained. $|S_{\lambda_1, \lambda_2}|$ denotes the number of nonzero coordinates in $\hat{\bm \beta}_{\lambda_1, \lambda_2}$, and the value $C_n = 6 \log (p)$ is recommended by \cite{Peng2015An} and \cite{Fan2021Penalized}.
By minimizing the HBIC criterion, we can effectively select the optimal $\lambda_1$ and/or $\lambda_2$ values for our nonconvex estimators. These choices allow us to balance the trade-off between model complexity and goodness of fit.

All experiments were performed using R on a computer equipped with an AMD Ryzen 9 7950X 16-Core Processor running at 4.50 GHz and with 32 GB RAM. To facilitate the implementation and usage of the LADMM algorithm, we have developed an R package called PIPADMM. The package is available at the following GitHub repository: \url{https://github.com/xfwu1016/PIPADMM}.

\subsection{Simulation for NPQR}\label{sec41}
In the first simulation, we apply the LADMM algorithm to solve the NPQR problem and compare its performance with several recent algorithms, including QRADM from \cite{Yu2017A}, QRADMslack from \cite{Fan2021Penalized}, and Feature-splitting ADMM (FSADMM) from \cite{Wen2023Feature}.
While both \cite{Yu2017A} and \cite{Fan2021Penalized} provided R packages for their respective algorithms, these packages are only compatible with Mac operating systems.
Additionally, the QPADM package only provide estimated coefficients and lacked information such as iteration count and iteration time. To ensure fairness in the comparison, we have rewritten the R code for the QPADM and QPADMslack algorithms, based on the descriptions provided in the respective papers.

For all tested ADMM algorithms, we set the maximum iteration number to 500, with the stopping criterion defined as follows:
\begin{align*}
\frac{\|\bm \beta^k - \bm \beta^{k-1} \|_2}{\max (1, \|\bm \beta^k \|_2)} \le 10^{-4}.
\end{align*}
This stopping criterion ensures that the difference between consecutive iterations of the estimated coefficients does not exceed a specified threshold.

Regarding our simulation studies, we employed the simulated models in the simulation studies of \cite{Peng2015An}, \cite{Yu2017A}, \cite{Fan2021Penalized} and \cite{Wen2023Feature}. Specifically, we generate data from the  heteroscedastic regression model  $y = x_6 + x_{12} + x_{15} + x_{20} + 0.7 x_1 \epsilon$, where $\epsilon \sim N(0,1)$. The covariates $(x_1,x_2,\dots,x_p)$ are generated in two steps.

\textbullet  \quad  First, we generate $\bm \tilde{x} = (\tilde{x}_1, \tilde{x}_2, \dots, \tilde{x}_p)^\top$ from a $p$-dimensional multivariate normal distribution $N(\bm 0, \bm \Sigma)$, where $\Sigma_{ij} = 0.5^{|i-j|}$ for $1 \le i,j \le p$.

\textbullet  \quad Second, we set $x_1 = \Phi(\tilde x_1)$ and $x_j = \tilde x_j$ for $j=2,\dots,p$.

In nonparallel environments ($M=1$), we simulate datasets with sizes ($n$, $p$) = (30,000, 1,000), (1,000, 30,000), (10,000, 30,000), and (30,000, 30,000). In parallel environments ($M \ge 2$), we simulate datasets with sizes ($n$, $p$) = (200,000, 500) and (500,000, 1,000). We run 500 independent simulations, and the average results for both nonparallel and parallel computations are presented in Table \ref{tab1} and Table \ref{tab2}, respectively. Due to space limitations, this section will only focus on the results of the SCAD ($a=3.7$) and $\tau = 0.7$, while the other simulation results are included in Appendix \ref{C.1}.

\begin{table}[H]
\caption{\footnotesize{Comparison of various ADMMs with SCAD penalty of different data sizes.}}\label{tab1}
\resizebox{1\columnwidth}{!}{
\begin{threeparttable}
\begin{tabular}{lcccccccccc}
\toprule[1.5pt]
\quad$(n,p)$ &  \multicolumn{5}{l}{(30000,1000)}   & \multicolumn{5}{l}{(1000,30000)}   \\
\cmidrule(lr){2-6}\cmidrule(lr){7-11}
             &  P1 &  P2 &  AE   &  Ite   &  Time  &  P1 &  P2 &  AE   &  Ite     &  Time  \\
\midrule[1pt]
QPADM        & 100 & 100 & \bf{0.016(0.03)} & 104.8(11.9)& 38.62(6.45)   & 0 & 100 & 7.585(0.91) & 500(0.0)  & 1924.5(55.3)\\
QPADMslack   & 100 & 100 & 0.058(0.03) & 51.2 (5.58)& 35.95(4.35)   & 0 & 100 & 8.465(1.12) & 216(13.5) & 1133.8(39.4)\\
FSADMM      & 100 & 100 & 0.017(0.04) & 500(0.00)  & 155.6(11.3)  &100& 100 & 0.621(0.09) & 500(0.00)  & 134.1(10.1)\\
LADMM        & 100 & 100 & 0.022(0.03) & \bf{44(2.29)}   & \bf{34.53(4.12)}   &100& 100 & \bf{0.482(0.05)} & \bf{166(11.9)} & \bf{60.7(5.0)}\\
\toprule[1.5pt]
\quad$(n,p)$ &  \multicolumn{5}{l}{(10000,30000)}  & \multicolumn{5}{l}{(30000,30000)}   \\
\cmidrule(lr){2-6}\cmidrule(lr){7-11}
             &  P1 &  P2 &  AE   &  Ite   &  Time  &  P1 &  P2 &  AE   &  Ite     &  Time  \\
\midrule[1pt]
QPADM        & 100 & 100 & 3.891(1.23) & 500(0.00)  & 2178.4(66.8) &100 & 100 & 1.758(0.13) & 500(0.00)   & 3264.3(56.3)\\
QPADMslack   & 100 & 100 & 4.171(1.64) & 252(20.7) & 1429.6(35.7) &100 & 100 & 2.263(0.25) & 327(21.1)  & 2126.3(41.2)\\
FSADMM      & 100 & 100 & 0.592(0.16) & 500(0.00)  & 238.9(18.1)  &100 & 100 & 1.092(0.10) & 500(0.00)   & 424.2(17.4)\\
LADMM        & 100 & 100 & \bf{0.348(0.08)} & \bf{182(13.6)} & \bf{75.6(6.6)}    &100 & 100 & \bf{0.663(0.09)} &\bf{ 239(15.2)}  & \bf{183.6(8.5)}\\
\toprule[1.5pt]
\end{tabular}
\begin{tablenotes}
        \footnotesize
        \item[*] The meanings of the notations used in this table are as follows: P1 (\%): proportion that $x_1$ is selected; P2 (\%): proportion that $x_6$, $x_{12}$, $x_{15}$, and $x_{20}$ are selected; AE: absolute estimation error; Ite: number of iterations; Time (s): running time. Numbers in the parentheses represent the corresponding standard deviations, and the optimal solution is represented in bold.
\end{tablenotes}
\end{threeparttable}}
\end{table}
Table \ref{tab1} provides evidence that LADMM outperforms QPADM and QPADMslack in terms of computational speed and estimation accuracy when the dimensionality of the problem $p$ is large. Furthermore, LADMM outperforms FSADMM when the sample size $n$ is large. To further illustrate the advantages of LADMM, we present the computational time in Figure \ref{fig3}. 
These numerical results indicate that LADMM demonstrates remarkable advantages in terms of computational efficiency and accuracy compared to other ADMM algorithms, particularly for scenarios where both \( n \) and/or \( p \) are large.
In Table \ref{tab2}, we observe that when multiple local machines are employed, LADMM shows comparable performances in terms of computational time and accuracy compared to QRADM and QRADMslack. This result is also visualized in Figure \ref{fig4}.  Table \ref{tab2} and Figure \ref{fig4}  underscore the advantages of utilizing LADMM in parallel computing settings. It is worth noting that our solution is not particularly influenced by $M$,  which is consistent with the theoretical results of Section \ref{sec32}.
\begin{table}[H]
\caption{\footnotesize{Comparison of various parallel ADMMs with SCAD penalty.}}\label{tab2}
\resizebox{1\columnwidth}{!}{
\begin{threeparttable}
\begin{tabular}{lcccccccc}
\toprule[1.5pt]
  &\multicolumn{2}{l}{QPADM}&\multicolumn{2}{r}{$(200000,500)$}&\multicolumn{2}{l}{QPADM}&\multicolumn{2}{r}{$(500000,1000)$} \\ 
\cmidrule(lr){2-5}\cmidrule(lr){6-9}
M & Nonzero &  AE   &  Ite   &  Time                        &  Nonzero &  AE   &  Ite   &  Time\\
\midrule[1pt]
 5  & 37.2(1.83) & \bf{0.053(0.0008)} & 342.8(23.9) & 72.8(5.63)   & 25.7(1.37) & 0.043(0.0007) & 458.3(35.8) & 182.8(10.5) \\
20  & 38.4(2.03) & 0.068(0.00113) & 375.1(29.2) & 39.8(4.73)   & 27.4(1.72) & 0.051(0.0009) & 483.1(39.2) & 99.5(5.90)  \\
100 & 40.6(2.3) & 0.085(0.00251) & 419.7(67.8) & 18.6(3.34)   & 33.5(2.15) & 0.077(0.0018) & 496.7(17.8) & 38.6(4.74)   \\
\toprule[1.5pt]
   & \multicolumn{2}{l}{QPADMslack}&\multicolumn{2}{r}{$(200000,500)$}&\multicolumn{2}{l}{QPADMslack}&\multicolumn{2}{r}{$(500000,1000)$} \\  
\cmidrule(lr){2-5}\cmidrule(lr){6-9}
M  &  Nonzero &  AE           &  Ite        &  Time         &  Nonzero &  AE   &  Ite   &  Time\\
\midrule[1pt]
 5  & 34.6(1.91) & 0.062(0.0017) & 214.4(25.4) & 31.2(4.05)   & 24.9(1.23) & 0.039(0.0006) & 311.7(28.2) & 79.5(7.93)\\
 20 & 39.7(2.15) & 0.075(0.0024) & 236.3(26.9) & 10.7(2.11)   & 29.5(1.50) & 0.048(0.0010) & 347.3(32.5) & 41.2(5.85)\\
100 & 41.2(2.55) & 0.094(0.0031) & 375.7(82.6) & \bf{4.61(0.83)}    & 32.3(1.91) & 0.074(0.0021) & 396.5(72.3) & 20.9(2.44)\\
\toprule[1.5pt]
   & \multicolumn{2}{l}{LADMM}&\multicolumn{2}{r}{$(200000,500)$}&\multicolumn{2}{l}{LADMM}&\multicolumn{2}{r}{$(500000,1000)$} \\  
\cmidrule(lr){2-5}\cmidrule(lr){6-9}
M  &  Nonzero &  AE           &  Ite        &  Time         &  Nonzero &  AE   &  Ite   &  Time\\
\midrule[1pt]
 5  & \bf{5.63(0.41)} & 0.055(0.0008) & \bf{49.7(1.21)}   & {19.7(1.02)}    & {5.42(0.35)} & \bf{0.032(0.0005)} & \bf{78.7(3.42)}   & {32.6(2.32)}\\
 20 & {5.64(0.41)} & {0.056(0.0008)} &{53.0(1.32)}  & {10.3(0.73)}    & {5.42(0.37)} & {0.036(0.0006)} &  {79.6(3.54)}   & {20.5(1.83)}\\
100 & {5.63(0.43)} & {0.055(0.0008)} & {56.9(1.38)}   & {5.32(0.45)}     &\bf{5.42(0.34)} & {0.035(0.0006)} & {81.0(3.60)}   & \bf{15.5(1.22)}\\
\toprule[1.5pt]
\end{tabular}
\begin{tablenotes}
        \footnotesize
        \item[*] P1 and P2 are not presented in Table  \ref{tab2} because all methods have a value of 100 for these two metrics. It is clear that Nonzero, AE, and Ite of LADMM are not significantly affected by the value of $M$. 
\end{tablenotes}
\end{threeparttable}}
\end{table}

\begin{figure}[htbp]
\centering
\includegraphics[width=17cm]{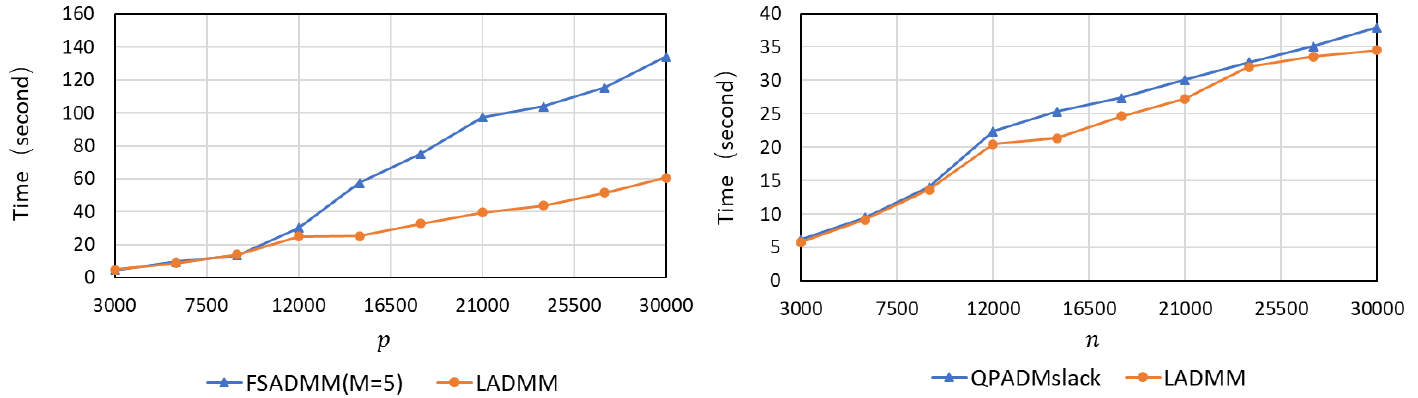}
\caption{\footnotesize{Time comparison of LADMM, QPADMslack and FSADMM. On the left, the dimension of the data is $n=500$; On the right, the sample size of the data is $p=500$.}}\label{fig3}
\end{figure}

\begin{figure}[htbp]
\centering
\includegraphics[width=17cm]{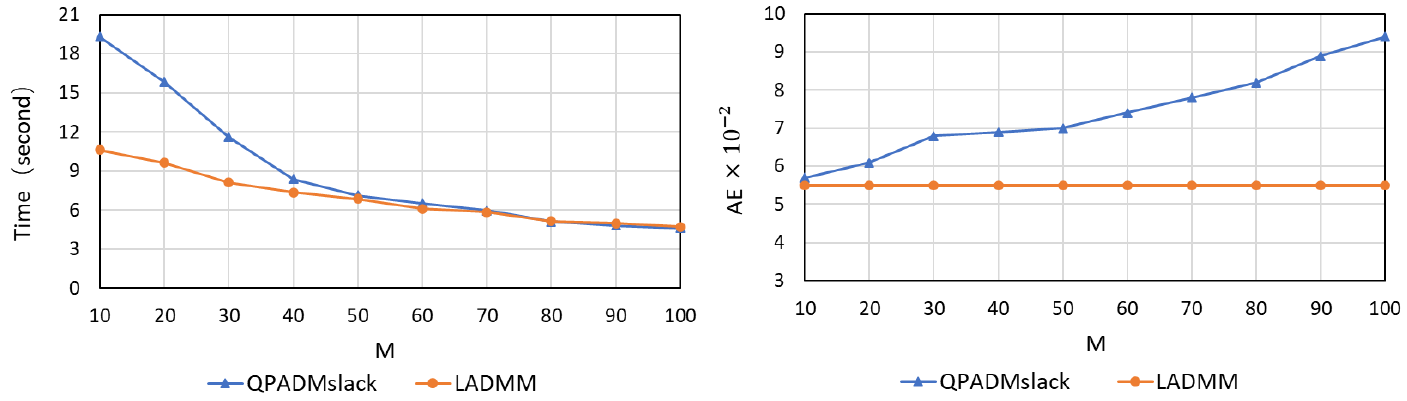}
\caption{\footnotesize{Computational time  and absolute error (AE) of QPADMslack and LADMM  with $n= 400000, p=1000$.}}\label{fig4}
\end{figure}

\subsection{Simulation for NPLS and NPHR}\label{sec42}
In this simulation, we utilize the proposed LADMM method and the ILAMM algorithm proposed in \cite{Fan2018I-LAMM} and \cite{Pan2021Iteratively} to solve nonconvex penalized least squares regression (NPLS) and nonconvex penalized Huber regression (NPHR). We then compare the performance of these methods. \cite{Pan2021Iteratively} provided  a R package for ILAMM  which  is available at \url{https://github.com/XiaoouPan/ILAMM}. 
Similar to \cite{Pan2021Iteratively}, we generate the heteroscedastic model, $y_i  = \bm x_i^\top \bm \beta + c^{-1}(\bm x_i^\top \bm \beta)^2  \epsilon_i$ with $\bm x_i \sim N(\bm 0,\bm I_p)$ for $i=1,\dots,n$, where the constant $c$ is chosen as $c=\sqrt 3 \bm \beta^\top \bm \beta$ such that  $E\left(c^{-1}(\bm x_i^\top \bm \beta)^2 \right)^2 = 1$. The true vector of regression coefficients $\bm \beta$ is $(4,3,2,-2,-2,-2,0,\dots,0)^\top$. Moreover, we consider the following two error distributions.
\begin{itemize}\small
    \item Normal distribution: $\epsilon_i \sim  N(\mu,\sigma^2)$ with $\mu = 0$ and standard deviation $\sigma=1.5$.
    \item  Lognormal distribution:  $\epsilon_i \sim  LN(\mu,\sigma^2)$  with $\mu = 0$ and  $\sigma=1.2$. 
\end{itemize}

\begin{table}\small
\caption{\footnotesize{Comparison of LS-SCAD and Huber-SCAD  Calculated by LADMM and ILAMM}}\label{tab3}
\resizebox{1\columnwidth}{!}{
\begin{threeparttable}
\begin{tabular}{lccccccccc}
\toprule[1.5pt]
         &           & \multicolumn{4}{l}{LS-SCAD} & \multicolumn{4}{l}{Huber-SCAD} \\
\cmidrule(lr){3-6}\cmidrule(lr){7-10}
 $\quad(n,p)$ & Algorithm & FP & FN & AE & Time         & FP & FN & AE & Time            \\
\midrule[1pt]
\multirow{2}*{(100,1000)}    & ILAMM & 0 & 1 & 0.34(0.08) & 8.8(1.2)    & 0 & 1 & 0.41(0.07) & 23.2(5.2)  \\
                             & LADMM & 0 & 0 & \bf{0.29(0.07)} & \bf{1.3(0.2)}    & 0 & 0 & \bf{0.34(0.07)} & \bf{1.41(0.3)}   \\
\multirow{2}*{(100,10000)}   & ILAMM & 0 & 0 & 0.50(0.10) & 175.3(10.3) & 0 & 0 & 0.52(0.18) & 362.9(15.9) \\
                             & LADMM & 0 & 0 & \bf{0.41(0.09)} & \bf{7.1(1.0)}    & 0 & 0 & \bf{0.49(0.16)} & \bf{7.8(0.8)}    \\
\multirow{2}*{(10000,10000)} & ILAMM & 0 & 1 & \bf{0.05(0.01)} & 1885.0(121.4)  & 0 & 1 & 0.07(0.01) & 2314(102.3)  \\
                             & LADMM & 0 & 0 & 0.06(0.02) & \bf{31.8(6.7)}  & 0 & 0 & \bf{0.01(0.00)} & \bf{35.2(7.5)}  \\
\toprule[1.5pt]
\end{tabular}
\begin{tablenotes}
        \footnotesize
        \item[*] False Positive (FP) refers to the number of variables with a coefficient of zero that are mistakenly included in the final model, while False Negative (FN) refers to the number of variables with non-zero coefficients that are omitted from the model.
\end{tablenotes}
\end{threeparttable}}
\end{table}
Table \ref{tab3} presents the simulation results of LADMM and ILAMM solving least squares SCAD (LS-SCAD) and Huber SCAD models (Huber-SCAD) when the errors follow a normal distribution. The results  demonstrate that LADMM outperforms the ILAMM in effectively solving both NPLS and NPHR models.  In addition, if $p$ is large, ILAMM may be time-consuming. 
The simulation results for the  Lognormal distribution, as well as other nonconvex penalties, are presented in the Appendix \ref{C.2}.

\subsection{Simulation for NPSQR}\label{sec43}
In this example,  we  utilize the proposed LADMM and algorithm in \cite{Mkhadri2017A}  to solve  nonconvex penalized smooth quantile regression (NPSQR)  and compare their performances. \cite{Mkhadri2017A} provided an efficient R package for calculating NPSQR, which can be found at the following \url{ https://github.com/KarimOualkacha/cdaSQR/tree/master}.  Following the scenario 3 in Section 4.2 of \cite{Mkhadri2017A}, we generate data sets with $y_i  = \bm x_i^\top \bm \beta  + \epsilon_i$,  where $\epsilon \sim N(0,\sigma^2 \bm I_p) $ and $\sigma = 5$. The covariates $(x_1,x_2,\dots,x_p)$ are generated from  $N(\bm 0, \bm \Sigma)$, where $\Sigma_{ij} = 0.5^{|i-j|}$ for $1 \le i,j \le p$. The true vector of regression coefficients is
$$\bm \beta = (  \underbrace {3,\dots,3}_{5}, \underbrace {-1.5,\dots,-1.5}_5,\underbrace {1,\dots,1}_5,\underbrace {2,\dots,2}_5,\underbrace {0,\dots,0}_{p-20})^\top.$$
The comparison results of our algorithm LADMM and cdaSQR are summarized in Table \ref{tab4}. These results are all about Snet, and the results for Mnet and Cnet are in the Appendix \ref{C.3}. In terms of computational accuracy and speed, LADMM and cdaSQR perform similarly, but LADMM has significant advantages in terms of FP and FN.

\begin{table}[H]
\caption{\footnotesize{Comparison of  LADMM and cdaSQR with $\delta=0.5$ and $\tau=0.5$}}\label{tab4}
\resizebox{1\columnwidth}{!}{
\begin{tabular}{lccccccccc}
\toprule[1.5pt]
         &                       & \multicolumn{4}{l}{cdaSQR}              & \multicolumn{4}{l}{LADMM} \\
\cmidrule(lr){3-6}\cmidrule(lr){7-10}
 $\quad(n,p)$ & Loss function              & FP & FN & AE & Time                     & FP & FN & AE & Time       \\
\midrule[1pt]
\multirow{2}*{(5000,500)}    & $\mathcal{L}_{\tau,c}$ & 2.52 & 7.39  & 0.47(0.05)  & \bf{3.36(0.36)} & \bf 0 & \bf{0.11}  & \bf{0.34(0.03)}  & 3.62(0.37) \\
                             & $\mathcal{L}_{\tau,\kappa}$  & 3.25 & 8.36  & 0.35(0.04)  & \bf{3.47(0.42)} & \bf 0 & \bf{0.08}  & \bf{0.29(0.02)}  & 3.78(0.38) \\
\multirow{2}*{(10000,1000)}  & $\mathcal{L}_{\tau,c}$ & 3.98 & 10.36 & 1.38(0.08)  & 5.11(0.82) & \bf 0 &  \bf 0.01  &  \bf \bf  0.88(0.05)  & \bf 4.71(0.67) \\
                             & $\mathcal{L}_{\tau,\kappa}$  & 2.52 & 9.47  & 1.22(0.08)  &\bf 4.76(0.71) & \bf  0 &  \bf  0.04  & \bf  \bf  \bf 0.74(0.04)  & 4.83(0.69) \\
\multirow{2}*{(20000,2000)}  &  $\mathcal{L}_{\tau,c}$ & 3.69 & 15.36 & 1.89(0.12)  & 24.33(2.39)&  \bf  0 & \bf  0.01  & \bf 0.96(0.09)  & \bf 20.39(2.48) \\
                             &  $\mathcal{L}_{\tau,\kappa}$  & 4.40 & 14.77 & 1.73(0.10)  & \bf 19.25(2.11)&  \bf 0 &  \bf 0.05  & \bf 0.89(0.08)  & 19.64(2.36) \\
\toprule[1.5pt]
\end{tabular}}
\end{table}

\section{Real Data Studies}\label{sec5}
In this section, we compare the performance of several parallel ADMM algorithms in an online publicly available dataset at \url{http://archive.ics.uci.edu/ml/
datasets/Online+News+Popularity}. This dataset provides a summary of the popularity, measured in terms of shares, as well as 60 features of 39,644 news published by Mashable over a two-year period. The features include various aspects such as binary variables indicating news categories (Lifestyle, Entertainment, Business, Social Media, Technology, or World), published time (day of the week and weekend or not), average word length, number of keywords, rate of nonstop words, and more.
For more information about the dataset, please refer to the study conducted by \cite{Fernandes2015A}.

This dataset is heterogeneous in nature, and it has been analyzed by \cite{Fan2021Penalized} using NPQR. Their aim was to analyze how various features impact the popularity of news, particularly focusing on those that had gained high levels of popularity. They  identified $x_{14}$ (Entertainment) and $x_{27}$  (number of key words) as the two most influential features. However, the empirical results in \cite{Fan2021Penalized} also indicated that  QPADM and QPADMslack algorithms tend to select a relatively larger number of variables, and this tendency becomes more pronounced as $M$ increases. 
The following empirical results indicate that LADMM does not have this increasing trend, and also has good prediction accuracy. 

As in the study by \cite{Fan2021Penalized}, we standardize the non-binary factors to have zero mean and unit variance. The features and response variables are denoted as $x_1, x_2, \dots, x_{60}$ and $y$, respectively. To assess the performance of the algorithm, we randomly partition the complete dataset 100 times. Each partition consisted of randomly selecting 35,000 samples as the training set, with the remaining samples designated as the test set.
To replicate the scenario of parallel computing, we randomly divide the training set into subsets of equal size. In our analysis, we consider values of $M=10$ and $M=100$. We then record the average number of selected nonzero coefficients (Nonzero), the prediction error (PE), the number of iterations (Ite), and the CPU running time (in seconds) of the algorithms.  The prediction error (PE) is calculated as follows: $\text{PE} = \frac{1}{n_{test}} \sum_{i=1}^{n_{test}} |y_i - \hat{y}_i|$, where $n_{test} = 4644$ represents the sample size of the test set. We only include the results of SCAD ($a=3.7$) regression in Table \ref{tab5}, and the results of MCP  regression are included in  Appendix \ref{C.4}.

\begin{table}[H]
\caption{\footnotesize{Analysis of the news popularity data under the SCAD penalty with $\tau=0.5$.}}\label{tab5}
\resizebox{1\columnwidth}{!}{
\begin{tabular}{lccccccc}
\toprule[1.5pt]
Algorithm                                    & M   & $x_{14}$      & $x_{27}$     & Nonzero     & PE         & Ite         & Time \\
\midrule[1pt]
\multirow{2}*{QRADM}                         & 10  & -0.021(0.002) & 0.054(0.002) & 34.53(1.47) & 0.22(0.02) & 223.6(37.8) & 20.1(2.32) \\
                                             & 100 & -0.020(0.002) & 0.056(0.001) & 34.82(1.25) & 0.20(0.02) & 294.8(50.3) & 4.5(0.67)  \\
\multirow{2}*{QRADMslack}                    & 10  & -0.019(0.001) & 0.055(0.001) & 33.92(1.59) & 0.20(0.03) & 112.9(11.2) & 9.8(1.03)  \\
                                             & 100 & -0.018(0.001) & 0.057(0.001) & 34.07(1.41) & 0.20(0.02) & 135.4(15.7) & 2.1(0.31)  \\
\multirow{2}*{$\text{LADMM}_{\rho}$}              & 10  & -0.010(0.001) & 0.021(0.001) & 10.34(0.56) & \bf 0.18(0.02) & 81.2(8.4)   & 3.2(0.21)  \\
                                             & 100 & -0.011(0.001) & 0.022(0.001) & 10.35(0.56) & \bf 0.18(0.02) & 83.6(8.7)   & 2.2(0.11)  \\
\multirow{2}*{$\text{LADMM}_{\mathcal L_{\tau,c},\delta=1}$} & 10  & -0.012(0.001) & 0.023(0.001) & 12.67(1.35) & 0.19(0.01) &\bf 65.9(9.6)   &\bf  2.9(0.13)  \\
                                             & 100 & -0.012(0.001) & 0.023(0.001) & 12.65(1.38) & 0.19(0.01) & \bf 68.3(9.7)   & \bf 2.1(0.05)  \\
\multirow{2}*{$\text{LADMM}_{\mathcal L_{\tau,\kappa},\delta=1}$} & 10  & -0.011(0.001) & 0.024(0.001) & 13.11(1.45) & 0.20(0.01) & 72.6(10.1)  & 3.0(0.21)  \\
                                             & 100 & -0.011(0.001) & 0.024(0.001) & 13.11(1.45) & 0.20(0.01) & 78.6(10.3)  & 2.1(0.08)  \\
\toprule[1.5pt]
\end{tabular}}
\end{table}

Table  \ref{tab5} demonstrates that  the LADMM algorithm exhibits competitive performance in terms of both time and accuracy compared to the QPADM and QPADMslack algorithms. Additionally, LADMM clearly demonstrates superior variable selection capability, and its iterative convergence is not affected by the value of $M$.
It is worth noting that, in addition to the important variables $x_{14}$ and $x_{27}$ identified in the previous study, our analysis reveals that variables $x_5$ (number of stopwords) and $x_{18}$ (a binary variable that represents whether the data channel is the word ``world" or not) are also significant. In particular, variable $x_5$ exhibits a notable importance. We visualize this result in Figure \ref{fig5}. Our newly discovered features are interpretable in real life. Excessive stopwords can make news verbose, trivial and lack coherence, making it difficult to convey clear ideas and information. On the other hand, the use of ``world" as data channels for news may result in a relatively smaller audience and, consequently, lower sharing volume, as the focus of news coverage might differ from the interests of the general public.
\begin{figure}[htbp]
\centering
\includegraphics[width=17cm]{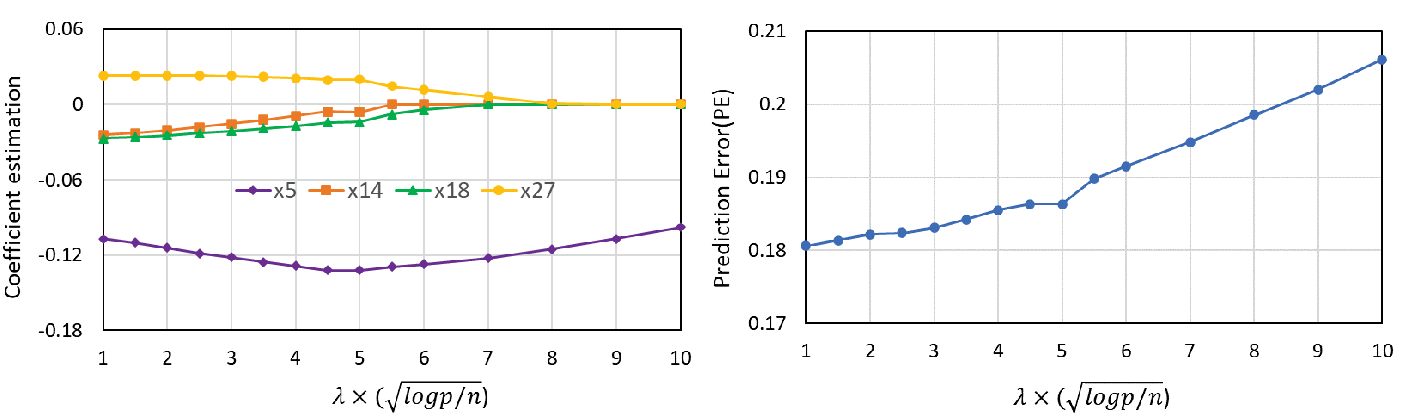}
\caption{\footnotesize{Variation of significance coefficient and  prediction error (PE) with $\lambda$}}\label{fig5}
\end{figure}

\section{Conclusion and Discussion}\label{sec6}
This paper introduces a parallel LADMM algorithm designed to efficiently solve nonconvex penalized smooth quantile regression problems. Our algorithm is easy to implement, and it offers great flexibility that allows for its extension to many other nonconvex penalized regression models, such as quantile regressions, least squares regressions, and Huber regressions. These extensions have also been implemented in our R package \textbf{PIPADMM}. At present, the parallel algorithm implementation in our R package is akin to the pseudo-parallelism of the R packages \textbf{QRADMM} and \textbf{QPADMslack}. One advantage of this pseudo-parallelism is that it is very convenient for readers to verify algorithms and reproduce code on a computer. Readers can also utilize our designed algorithm framework for processing distributed stored data in Spark. Notably, compared to existing parallel algorithms that rely on consensus-based approaches for solving regression models,  our solution companion changes very little with different sample partitioning strategies and requires fewer iterations to achieve convergence. Furthermore, for smooth loss functions, we prove the global convergence of our LADMM algorithm, meaning that the iterative solutions converge to a critical point of the Lagrangian function.

The parallel algorithm design process outlined in Section \ref{sec32}  indicates that other linearized ADMM algorithms, such as those presented in \cite{Li2014Linearized}, \cite{Gu2018ADMM}, and \cite{Liang2024Linearized}, can be readily adapted to parallel computation using our strategy.
Although this paper only deals with regression problems with simple combined regularization like Elastic-net, Snet, Mnet, and Cnet, our algorithm also has the potential to handle more complex combinatorial regularization terms. Examples include sparse group lasso in \cite{Wu2008Coordinate} and \cite{Juan2019An}, as well as sparse fused lasso in \cite{Tibshirani2005Sparsity} and \cite{Wu2024Multi}.

There is still much work to be done, such as studying the convergence of parallel ADMM algorithms for solving NPQR, which has remained an unresolved open problem in \cite{Yu2017A}. This paper only employs NPSQR to approximate NPQR, which appears to solve the convergence issue. However, this study demonstrated that the global convergence claimed by Theorem \ref{TH2} cannot be guaranteed for arbitrarily small $c$ and $\kappa$. Furthermore, our algorithm not only efficiently handles regularized regression problems, but also easily extends to regularized classification problems with smoothing losses, such as Huber loss SVM in \cite{Wang2008Hybrid} and least squares loss SVM in \cite{Huang2014Asymmetric}.

\section*{Acknowledgements}
We are grateful to Professor Karim Oualkacha from the Department of Mathematics at the Université du Québec à Montréal for providing us with the code for the R package cdaSQR. Specifically, we are very grateful to Professor Bingsheng He for his valuable discussions with us, which greatly helped us prove the convergence of the algorithm. The research was partially supported by the National Natural Science Foundation of China [grant numbers 11871121, 12171405, 12271066] and the project of science and technology research program of Chongqing Education Commission of China [Grant Numbers KJQN202302003].

\bibliographystyle{apalike}
\bibliography{myrefq}

\newpage
\section*{Online support materials}
\appendix
In this online support materials, we first provide the derivation of the closed-form solution for the proximal operator in Section \ref{sec23} of the main text. Next, we give the proofs of  Theorem \ref{TH1} and Theorem \ref{TH2}. Finally, we present supplementary numerical experiments.
\section{Proof of the Closed-form Solution of Proximal Operator}\label{A}
In \ref{A1}, we derive closed-form solutions for the proximal operators of two smooth quantile regression methods similar to those in \cite{Liang2024Linearized}. In \ref{A2}, we propose a simple and scalable approach for deriving closed-form solutions of the proximal operator for regularized terms.
\subsection{Loss Fuction}\label{A1}
Recalling that the loss function ${\mathcal{L}_{\tau,*}}(\bm u)= \sum_{i=1}^{n}{\mathcal{L}_{\tau,*}}(u_i)$, we can divide the optimization problem
\begin{align}\label{2e}
\mathop{\arg \min}\limits_{\bm u} \left\{{\mathcal{L}_{\tau,*}}(\bm u) + \frac{\mu}{2}\left\|\bm u - \bm v\right\|_2^2\right \}
\end{align}
 into $n$ independent univariate problems, which can be expressed as $\mathop{ \min}\limits_{u_i} \left\{  {\mathcal{L}_{\tau,*}}(u_i) + \frac{\mu}{2}\left(u_i - v_i\right)_2^2\right \} , i=1,2,\dots,n$.  We denote the  minimum solution of  (\ref{2e}) by $\hat{\bm u}$.

\textbullet  \quad $\mathcal{L}_{\tau,c}(u_i)$: Let us examine the expression for $\mathcal{L}_{\tau,c}(u_i)$:
\begin{align}
 \begin{cases} 
\tau (u_i - 0.5c) & \text{if } u_i \geq c, \\
\frac{\tau u_i^2}{2c} & \text{if } 0 \leq u_i < c, \\
\frac{(1- \tau) u_i^2}{2c} & \text{if } -c \leq u_i < 0, \\
(\tau -1)(u_i +0.5c) & \text{if } u_i < -c.
\end{cases}
\end{align}
Throughout the entire domain of the above function, it is either a linear or quadratic function. Therefore, it has a closed-form solution in each small domain. After some algebra, we have 
\begin{align}
\hat{ u}_j  = \begin{cases}
 v_j -  \frac{\tau}{\mu}, & \text{{if }}  v_j \geq c + \frac{\tau}{\mu}, \\
\frac{c\mu v_j }{c\mu + \tau}, & \text{{if }} 0 \le  v_j < c + \frac{\tau}{\mu}, \\
\frac{c\mu v_j }{c\mu + 1 - \tau}, & \text{{if }}  - c + \frac{\tau - 1}{\mu} \le  v_j <  0, \\
 v_j -  \frac{\tau -1}{\mu}, & \text{{if }}  v_j <  - c + \frac{\tau - 1}{\mu}. \\
\end{cases}
\end{align}

\textbullet  \quad $\mathcal{L}_{\tau,\kappa}(u_i)$: Let us examine the expression for $\mathcal{L}_{\tau,\kappa}(u_i)$:
\begin{align}
 \begin{cases} 
\tau (u_i - \frac{\tau\kappa}{2} )& \text{if } u_i > \tau\kappa, \\
\frac{u_i^2}{2\kappa} & \text{if}  \  u_i \in [(\tau-1)\kappa, \tau\kappa] , \\
(\tau - 1) [u_i - \frac{(\tau-1)\kappa}{2}]& \text{if}  \  u_i  <  (\tau-1)\kappa.
\end{cases}
\end{align}
Similar to  $\mathcal{L}_{\tau,c}(u_i)$,  we can derive the following  closed-form solution of $\mathcal{L}_{\tau,\kappa}(u_i)$, 
\begin{align}
\hat{u}_j   = \begin{cases}
 v_j -  \frac{\tau}{\mu}, & \text{{if }}  v_j \geq \tau \kappa + \frac{\tau}{\mu}, \\
\frac{c\mu v_j }{c\mu + \tau}, & \text{{if }}  (\tau -1) \kappa + \frac{\tau - 1}{\mu} \le v_j < \tau \kappa + \frac{\tau}{\mu}, \\
 v_j -  \frac{\tau -1}{\mu}, & \text{{if }} v_j <   (\tau -1) \kappa + \frac{\tau - 1}{\mu} \\
\end{cases}
\end{align}

\subsection{Nonconvex Penalty}\label{A2}
Recalling that the nonconvex penalty $P_\lambda (|\bm u|) = \sum_{j}^{p}P_\lambda(|u_j|)$, we can divide the optimization problem $$\mathop{\arg \min}\limits_{\bm u} \left\{ P_\lambda(|\bm u|) + \frac{\eta}{2}\left\|\bm u - \bm v\right\|_2^2\right \}$$ into $p$ independent univariate problems, which can be expressed as $\mathop{ \min}\limits_{u_j} \left\{ P_\lambda(|u_j|) + \frac{\mu}{2}\left(u_j - v_j\right)_2^2\right \} , j=1,2,\dots,p$. 
Let 
\begin{align}\label{A21}
\hat{u}_j = \mathop{\arg \min}\limits_{u_j} \left\{ P_\lambda(|u_j|) +  \frac{\eta}{2}\left(u_j - v_j\right)_2^2\right \}.
\end{align}
Because $ P_\lambda(|u_j|)$ is a function of the absolute value of $u_j$, and the subsequent term is a quadratic function ($v_j$ is a constant), we have $\text{sign}(\hat u_j) = \text{sign}(v_j)$ and $\text{sign}(v_j) \hat u_j \ge 0$. Let 
\begin{align}\label{A22}
\tilde{u}_j = \text{sign}(v_j) \hat u_j,
\end{align}
 then we can convert (\ref{A21}) into the following optimization formula
\begin{align}\label{s1}
\tilde{u}_j = 
 \mathop{\arg \min}\limits_{u_j \ge 0} \left\{ P_\lambda(u_j) +\frac{\eta}{2}\left(u_j -|v_j| \right)_2^2\right \}.
\end{align}
Make the first derivative of the optimization formula above equal to $0$, resulting in
 \begin{align}\label{ke}
\tilde{u}_j   = |v_j| - \nabla P_\lambda(\tilde{u}_j )/{\eta}  \ \text{and} \ \tilde{u}_j \ge 0. 
\end{align}

By (\ref{A22}), we have
\begin{align}\label{ke2}
 \hat{u}_j  =  \text{sign}(v_j)\tilde{u}_j.   
\end{align}

The above two equations play \textbf{a crucial role} in deriving the closed-form solutions for the proximal operators of the (nonconvex) penalty term. Clearly,  $\hat{u}_j$ in (\ref{ke2}) is the solution of the  proximal operator.

$\bullet$ For Snet, we have \begin{align}\label{ssnet}
\nabla P_\lambda(\tilde{u}_j ) = \begin{cases}
\lambda_1 + \lambda_2 \tilde{u}_j , & \text{{if }} \tilde{u}_j  \leq \lambda_1, \\
\frac{a\lambda_1 -  \tilde{u}_j}{a-1} + \lambda_2 \tilde{u}_j , & \text{{if }} \lambda_1 < \tilde{u}_j  < a  \lambda_1, \\
\lambda_2 \tilde{u}_j , & \text{{if }} \tilde{u}_j  \ge a  \lambda_1. \\
\end{cases}
\end{align}
By substituting (\ref{ssnet}) into the equation (\ref{ke}), we can discuss the solutions of the equation  (\ref{ke}) in different regions.
For $ \tilde{u}_j  \leq \lambda_1$,  the solution of equation (\ref{ke})  is  $ \left[ \frac{\eta| v_j| - \lambda_1}{\eta + \lambda_2} \right]_+$. For $\lambda_1 < \tilde{u}_j  < a  \lambda_1$,  the solution of equation (\ref{ke})  is  $\left[\frac{(a-1)\eta| v_j| - a \lambda_1}{(a-1) (\eta + \lambda_2) - 1}\right]_+$. For $\tilde{u}_j  \ge a \lambda_1$,   the solution of equation (\ref{ke})  is $\frac{\eta |v_j|}{\eta + \lambda_2}$. Note that  $ \hat{u}_j = \text{sign}(v_j) \tilde{u}_j $ and $(a-1)(\eta + \lambda_2) > 1$, it follows that
\begin{align}
\hat{u}_j   = \begin{cases}
\text{{sign}}(v_j) \cdot \left[ \frac{\eta| v_j| - \lambda_1}{\eta + \lambda_2} \right]_+, & \text{{if }} | v_j| \leq \frac{\lambda_1(1+\eta+\lambda_2)}{\eta}, \\
\text{{sign}}( v_j) \cdot \left[ \frac{(a-1)\eta| v_j| - a \lambda_1}{(a-1) (\eta + \lambda_2) - 1} \right]_+, & \text{{if }} \frac{\lambda_1(1+\eta+\lambda_2)}{\eta} < | v_j| < \frac{a \lambda_1(\eta+\lambda_2)}{\eta}, \\
\frac{\eta v_j}{\eta + \lambda_2}, & \text{{if }} | v_j| \geq  \frac{a \lambda_1(\eta+\lambda_2)}{\eta}. \\
\end{cases}
\end{align}

$\bullet$ For Mnet, we have \begin{align}\label{mmnet}
\nabla P_\lambda(\tilde{u}_j )  =
\begin{cases}
\lambda_1  - \frac{\tilde{u}_j }{a}  + \lambda_2 \tilde{u}_j, & \text{{if }}  \tilde{u}_j  \le a \lambda_1,\\
\lambda_2 \tilde{u}_j, & \text{{if }}  \tilde{u}_j  > a \lambda_1. \\
\end{cases}
\end{align}
By substituting (\ref{mmnet}) into the equation (\ref{ke}), we can discuss the solutions of the equation   (\ref{ke}) in different regions.
For $\tilde{u}_j  \le a \lambda_1$,  the solution of equation (\ref{ke})  is  $ [\frac{a \eta| v_j| - a \lambda_1}{a(\eta + \lambda_2) -1}]_+$.  For $\tilde{u}_j  > a \lambda_1$, the solution of equation (\ref{ke})  is  $\frac{\eta | v_j|}{\eta + \lambda_2}$.
Together with  $ \hat{u}_j = \text{sign}(v_j) \tilde{u}_j $,   we obtain
\begin{align}
\hat{u}_j  = \begin{cases}
\text{{sign}}( v_j) \cdot \left[ \frac{a \eta| v_j| - a \lambda_1}{a(\eta + \lambda_2) -1} \right]_+, & \text{{if }} | v_j| <  \frac{a \lambda_1(\eta+\lambda_2)}{\eta}, \\
\frac{\eta  v_j}{\eta + \lambda_2}, & \text{{if }} | v_j| \geq  \frac{a \lambda_1(\eta+\lambda_2)}{\eta}. \\
\end{cases}
\end{align}

$\bullet$ For Cnet, we have \begin{align}\label{ccnet}
\nabla P_\lambda(\tilde{u}_j )  =
\begin{cases}
\lambda_1   + \lambda_2 \tilde{u}_j, & \text{{if }}  \tilde{u}_j  \le a,\\
\lambda_2 \tilde{u}_j, & \text{{if }}  \tilde{u}_j  > a. \\
\end{cases}
\end{align}
By substituting (\ref{ccnet}) into the equation (\ref{ke}), we can discuss the solutions of the equation  (\ref{ke})  in different regions.
For $\tilde{u}_j  \le a$,  the solution of equation (\ref{ke})  is  $[\frac{ \eta| v_j| -  \lambda_1}{\eta + \lambda_2}]_+$. For $\tilde{u}_j  > a \lambda_1$, the solution of equation (\ref{ke})  is  $\frac{\eta  |v_j|}{\eta + \lambda_2}$. Together with  $ \hat{u}_j = \text{sign}(v_j) \tilde{u}_j $,  we have
\begin{align}
 \hat{u}_j = \begin{cases}
\text{{sign}}( v_j) \cdot \left[ \frac{ \eta| v_j| -  \lambda_1}{\eta + \lambda_2} \right]_+, & \text{{if }} | v_j| <  \frac{a (\eta+\lambda_2)}{\eta}, \\
\frac{\eta v_j}{\eta + \lambda_2}, & \text{{if }} |v_j| \geq  \frac{a (\eta+\lambda_2)}{\eta}. \\
\end{cases}
\end{align}

\section{Proof of Theorem \ref{TH1} and \ref{TH2}}\label{B}
\subsection{Proof of Theorem \ref{TH1}}\label{B1}
Looking back at (\ref{Xy}) and (\ref{rb}), we have
\begin{align*}
{ \left[ \mu \bm X^\top(\bm{X\beta}^k + \bm{r}^k - \bm y - \bm d^k/\mu)\right]} = \left[ \mu \sum_{m=1}^{M} \bm X_m^\top(\bm{X}_m \bm \beta^k + \bm{r}_m^k - \bm y_m - \bm d_m^k/\mu) \right].
\end{align*}
If the two algorithms have the same $k$-th iteration solution $\left\{ {\bm \beta}^k, {\bm r}^k, {\bm d}^k\right\}$, then we can derive the following equation from (\ref{uadmm}) and (\ref{Mupdate}),
\begin{align}\label{bb}
\hat{\bm \beta}^{k+1} =  \tilde{\bm \beta}^{k+1}, k=0,1,2,3,\dots.
\end{align}
Substituting the above equation into (\ref{twoupladmm1}) and (\ref{twoupladmm2}), along with (\ref{rb}), we can observe that $\hat{\bm r}^{k+1} = \tilde{\bm r}^{k+1}$ and $\hat{\bm d}^{k+1} = \tilde{\bm d}^{k+1}$. Therefore, we are able to conclude the result of Theorem \ref{TH1}.

\subsection{Proof of Theorem \ref{TH2}}\label{B2}
The proof of Theorem  \ref{TH2} mainly follows the proof of algorithm convergence presented in  \cite{Guo2016Convergence}. However, there are two key differences in our proof. One is the inclusion of an additional quadratic linearization term in our algorithm.  The other is our method does not rely on assumptions about the boundedness of iterative solutions, which is necessary for the convergence proof in \cite{Guo2016Convergence}.
Our proof can be divided into the following four steps,
\begin{itemize}\small
    \item  First, we prove that the augmented Lagrangian function that requires alternating iteration minimization is monotonically nonincreasing.

    \item  Second, we demonstrate that the iterative solutions of our algorithm remain bounded throughout the entire iteration process.

\item Third, we demonstrate that the iterative solutions of LADMM satisfy $\lim_{k \rightarrow +\infty} \|\mathbf{w}^k - \mathbf{w}^{k+1} \|_2 = 0$. 
 
\item Finally, we prove that the augmented Lagrangian function of all limit points of our LADMM algorithm is a constant.
\end{itemize}
Upon completing these four steps, following the approach outlined in Theorem 3.1 of \cite{Guo2016Convergence}, we can conclude that the iterative sequence produced by our algorithm is a Cauchy sequence.

Now, we start our proof. Because Theorem \ref{TH1} has been proven to take the same $\eta$, the solutions of Algorithm \ref{alg1} and Algorithm \ref{alg2} are exactly the same. Therefore, we only need to discuss the nonparallel version of the LADMM algorithm. In terms of constrained optimization, we have 
\begin{align}\label{pb1}
\min_{\bm \beta, \bm r} & \quad  \sum_{i=1}^{n} \mathcal{L}_{\tau,*}(r_i) + {P}_\lambda(|\bm \beta|), \notag \\
\text{s.t.} & \quad \bm X \bm \beta + \bm r = \bm y. 
\end{align}
In \cite{Guo2016Convergence} , they assumed that $\mathcal{L}_{\tau,*}$ is a continuously differentiable function with Lipschitz continuous gradient $\nabla \mathcal{L}_{\tau,*}$. Here, we employ the following lemma to ensure that $\mathcal{L}_{\tau,c}$ and $\mathcal{L}_{\tau,\kappa}$ possess this property.
\begin{lem}\label{lem1}(Proposition 2 in \cite{Mkhadri2017A})
The smooth quantile $\mathcal{L}_{\tau,c}$ and $\mathcal{L}_{\tau,\kappa}$ are differentiable
and have Lipschitz continuous first derivatives, that is,
\begin{align*}
 \|\nabla \mathcal{L}_{\tau,c}(u_1) - \nabla \mathcal{L}_{\tau,c}(u_2)\|_2 & \le \frac{\max \{\tau,1-\tau \}}{c}\|u_1 - u_2  \|_2, \\
  \|\nabla \mathcal{L}_{\tau,\kappa}(u_1) - \nabla \mathcal{L}_{\tau,\kappa}(u_2)\|_2 & \le \frac{1}{\kappa}\|u_1 - u_2  \|_2 ,
\end{align*}
where $u_1$ and $u_2$ are two arbitrary real numbers.
\end{lem}

Problem (\ref{pb1}) has the following augmented Lagrangian,
\begin{equation}\label{ALf}
{L_\mu }(\bm \beta, \bm r, \bm d) = \mathcal{L}_{\tau,*}(\bm r) +  {P}_\lambda(|\bm \beta|) - {\bm d^\top}(\bm{X\beta} + \bm{r} - \bm y) + \frac{\mu }{2}\left\|\bm{X\beta} + \bm{r} - \bm y \right\|_2^2,
\end{equation} 
and the iterative steps of LADMM are as follows,
\begin{equation}\label{twoupladmmf}
\left\{ \begin{array}{l}
\bm \beta^{k+1} \ =\mathop {\arg \min }\limits_{\bm \beta} \left\{  {P}_\lambda(|\bm \beta|) + \frac{\mu }{2}\left\|\bm{X\beta} + \bm{r}^k - \bm y - \bm d^k/\mu \right\|_2^2 +   \frac{1}{2}\left \| \bm \beta - \bm \beta^k \right \|_{\bm S}^2    \right\},\\
\bm r^{k+1} \ =\mathop {\arg \min }\limits_{\bm r} \left\{ \mathcal{L}_{\tau,*}(\bm r) +  \frac{\mu }{2}\left\|\bm{X}\bm{\beta}^{k+1} + \bm{r} - \bm y - \bm d^k/\mu \right\|_2^2 \right\},\\
\bm{d}^{k+1}=\bm{d}^{k}-\mu(\bm{X}\bm{\beta}^{k+1} + \bm{r}^{k+1} - \bm y).
\end{array} \right.
\end{equation}
We first demonstrate that the augmented Lagrangian function $ L_\mu(\bm \beta^k, \bm r^k, \bm d^k)$ decreases as the number of iterations increases.

\begin{lem}\label{lem2}
Let the sequence $\boldsymbol{w}^{k}=\{\boldsymbol{\beta}^k, \boldsymbol{r}^k, \boldsymbol{d}^k\}$ be generated by the LADMM, then
\begin{align}\label{key}
L_\mu(\bm w^{k+1})  \le  L_\mu(\bm w^{k}) + \left(\frac{n}{\mu  \min\{c^2,\kappa^2 \}} - \frac{\mu}{2}\right)\|\bm r^k - \bm r^{k+1} \|_2^2.
\end{align}
Together with $\mu >  \frac{\sqrt{2n}}{\min\{c,\kappa \}}$,  $\{L_\mu(\bm w^k) \}_{k=1}^{\infty}$ is monotonically nonincreasing.
\end{lem}

\begin{proof}
 The  optimality conditions for the $k+1$-iteration of LADMM is 
\begin{equation}\label{oc}
\left\{ \begin{array}{l}
\bm 0 \ \in \partial{P}_\lambda(|\bm \beta^{k+1}|) - \bm X^\top \bm d^k + \mu \bm X^\top (\bm{X} \bm{\beta}^{k+1} + \bm{r}^k - \bm y) + \bm S(\bm \beta^{k+1} - \bm \beta^k),\\
\bm 0 \ = \nabla \mathcal{L}_{\tau,*}(\bm r^{k+1}) - \bm d^k + \mu  (\bm{X} \bm{\beta}^{k+1} + \bm{r}^{k+1} - \bm y) ,\\
\bm{d}^{k+1}=\bm{d}^{k}-\mu(\bm{X}\bm{\beta}^{k+1} + \bm{r}^{k+1} - \bm y).
\end{array} \right.
\end{equation}
Using the last equality and rearranging terms, we obtain
\begin{equation}\label{oc2}
\left\{ \begin{array}{l}
\bm 0 \ \in \partial{P}_\lambda(|\bm \beta^{k+1}|) - \bm X^\top \bm d^{k+1} + \mu \bm X^\top (\bm{r}^k - \bm r^{k+1}) + \bm S(\bm \beta^{k+1} - \bm \beta^k),\\
\nabla \mathcal{L}_{\tau,*}(\bm r^{k+1})  \ =  \bm d^{k+1}.
\end{array} \right.
\end{equation}
From the definition of the augmented Lagrangian function $L_\mu$, it follows that
\begin{align}\label{oc3}
L_\mu(\boldsymbol{\beta}^{k+1}, \boldsymbol{r}^{k+1}, \boldsymbol{d}^{k+1}) & = L_\mu(\boldsymbol{\beta}^{k+1}, \boldsymbol{r}^{k+1}, \boldsymbol{d}^{k}) + \langle \bm d^k - \bm d^{k+1}, \bm X \bm \beta^{k+1} + \bm r^{k+1} - \bm y \rangle   \notag\\
& =  L_\mu(\boldsymbol{\beta}^{k+1}, \boldsymbol{r}^{k+1}, \boldsymbol{d}^{k})  + \frac{1}{\mu}\|\bm d^k - \bm d^{k+1} \|_2^2,
\end{align}
and 
\begin{align}\label{eq1}
L_\mu(\boldsymbol{\beta}^{k+1}, \boldsymbol{r}^{k}, \boldsymbol{d}^{k}) & - L_\mu(\boldsymbol{\beta}^{k+1}, \boldsymbol{r}^{k+1}, \boldsymbol{d}^{k})  = (\mathcal{L}_{\tau,*}(\bm r^k) - \mathcal{L}_{\tau,*}(\bm r^{k+1}) ) - \langle \bm d^k,  \bm r^{k} - \bm r^{k+1} \rangle \notag\\
&   +  \frac{\mu}{2} \left[ \| \bm X \bm \beta^{k+1} + \bm r^{k} - \bm y \|_2^2 - \| \bm X \bm \beta^{k+1} + \bm r^{k+1} - \bm y \|_2^2 \right].
\end{align}
Since $\mathcal{L}_{\tau,*}$ is a convex function and $\bm{d}^{k+1}=\bm{d}^{k}-\mu(\bm{X}\bm{\beta}^{k+1} + \bm{r}^{k+1} - \bm y)$, we have
\begin{equation}\label{eq2}
\left\{ \begin{array}{l}
\mathcal{L}_{\tau,*}(\bm r^k) - \mathcal{L}_{\tau,*}(\bm r^{k+1}) \ge  \langle \nabla \mathcal{L}_{\tau,*}(\bm r^{k+1}), \bm r^k - \bm r^{k+1} \rangle,\\
\bm X \bm \beta^{k+1} + \bm r^{k} - \bm y = (\bm d^k - \bm d^{k+1})/\mu + (\bm r^k - \bm r^{k+1}) ,\\
\bm X \bm \beta^{k+1} + \bm r^{k+1} - \bm y = (\bm d^k - \bm d^{k+1})/\mu.
\end{array} \right.
\end{equation}
Note that $\nabla \mathcal{L}_{\tau,*}(\bm r^{k+1})   =  \bm d^{k+1}$, and inserting Equation (\ref{eq2}) into Equation (\ref{eq1}) yields
\begin{align}\label{eq3}
 L_\mu(\boldsymbol{\beta}^{k+1}, \boldsymbol{r}^{k+1}, \boldsymbol{d}^{k}) \le  L_\mu(\boldsymbol{\beta}^{k+1}, \boldsymbol{r}^{k}, \boldsymbol{d}^{k}) - \frac{\mu}{2}\|\bm r^k - \bm r^{k+1} \|_2^2.
\end{align}
Combining Equations (\ref{oc3}) and (\ref{eq3}), we get
\begin{align}
L_\mu(\boldsymbol{\beta}^{k+1}, \boldsymbol{r}^{k+1}, \boldsymbol{d}^{k+1}) \le L_\mu(\boldsymbol{\beta}^{k+1}, \boldsymbol{r}^{k}, \boldsymbol{d}^{k})  - \frac{\mu}{2}\|\bm r^k - \bm r^{k+1} \|_2^2 + \frac{1}{\mu}\|\bm d^k - \bm d^{k+1} \|_2^2.
\end{align}
From Lemma \ref{lem1} and recalling that $ \mathcal{L}_{\tau,*}(\bm r) = \sum_{i=1}^{n} \mathcal{L}_{\tau,*}(\bm r_i) $, we can derive
\begin{align}\label{rr}
 \|\bm d^k - \bm d^{k+1} \|_2 = \|\nabla \mathcal{L}_{\tau,*}(\bm r^k) - \nabla \mathcal{L}_{\tau,*}(\bm r^{k+1})\|_2  \le \frac{\sqrt n}{ \min\{c,\kappa\}}\|\bm r^k - \bm r^{k+1} \|_2.
\end{align}
Consequently, we have
\begin{align}
L_\mu(\boldsymbol{\beta}^{k+1}, \boldsymbol{r}^{k+1}, \boldsymbol{d}^{k+1}) \le L_\mu(\boldsymbol{\beta}^{k+1}, \boldsymbol{r}^{k}, \boldsymbol{d}^{k}) + \left(\frac{n}{\mu  \min\{c^2,\kappa^2 \}} - \frac{\mu}{2}\right)\|\bm r^k - \bm r^{k+1} \|_2^2.
\end{align}
Note that $\bm \beta^{k+1} $ is  the value that minimizes $L_\mu(\bm \beta, \bm r^k, \bm d^k) + \frac{1}{2}\left \| \bm \beta - \bm \beta^k \right \|_{\bm S}^2$ and $\bm S$ is a positive-definite matrix, then we get
\begin{align}
L_\mu(\bm \beta^{k+1}, \bm r^k, \bm d^k) \le L_\mu(\bm \beta^{k}, \bm r^k, \bm d^k).
\end{align}
As a result, 
\begin{align}
L_\mu(\boldsymbol{\beta}^{k+1}, \boldsymbol{r}^{k+1}, \boldsymbol{d}^{k+1}) \le L_\mu(\boldsymbol{\beta}^{k}, \boldsymbol{r}^{k}, \boldsymbol{d}^{k}) + \left(\frac{n}{\mu  \min\{c^2,\kappa^2 \}} - \frac{\mu}{2}\right)\|\bm r^k - \bm r^{k+1} \|_2^2.
\end{align}
Since we assume that $\mu >  \frac{\sqrt{2n}}{\min\{c,\kappa \}}$, then we have  $\left(\frac{n}{\mu \min\{c^2,\kappa^2 \}} - \frac{\mu}{2}\right)<0$, which implies that $\{L_\mu(\bm w^k) \}$ is monotonically nonincreasing. $\blacksquare$
\end{proof}

Next, we will prove that the iterative sequences generated by LADMM are bounded.

\begin{lem}\label{lem3}
(a) With $\mu >  \frac{\sqrt{2n}}{\min\{c,\kappa \}}$ and  $\bm X^\top \bm X > \underline{\mu} \bm I_p$$ (\underline{\mu} > 0)$, if $P_\lambda(\bm \beta)$ is the SCAD,MCP or Capped-$\ell_1$, then the sequence $\boldsymbol{w}^{k}=\{\boldsymbol{\beta}^k, \boldsymbol{r}^k, \boldsymbol{d}^k\}$ is generated by LADMM is  bounded.  \\

\qquad  \quad \ (b) With $\mu >  \frac{\sqrt{2n}}{\min\{c,\kappa \}}$,  if    ${P}_\lambda(|\bm \beta^k|)$ is Mnet, Snet or Cnet, then the sequence $\boldsymbol{w}^{k}=\{\boldsymbol{\beta}^k, \boldsymbol{r}^k, \boldsymbol{d}^k\}$ is generated by LADMM is  bounded.
\end{lem}
\begin{proof}
Since $\{L_\mu(\bm w^k) \}$ is monotonically nonincreasing, we have
\begin{align}
L_\mu(\bm w^0) \ge L_\mu(\bm w^k) &= \mathcal{L}_{\tau,*}(\bm r^k) +  {P}_\lambda(|\bm \beta^k|) -\langle { \bm d^k}, \bm{X\beta}^k + \bm{r^k} - \bm y \rangle + \frac{\mu }{2}\left\|\bm{X\beta}^k + \bm{r}^k - \bm y \right\|_2^2  \notag\\
& =  \mathcal{L}_{\tau,*}(\bm r^k) +  {P}_\lambda(|\bm \beta^k|) + \frac{\mu }{2}\left\|\bm{X\beta}^k + \bm{r}^k - \bm y  - \bm d^k/\mu \right\|_2^2 - \frac{\|\bm d^k \|_2^2}{2\mu}.
\end{align}
Because the value of  $\bm w^0$  is given, $L_\mu(\bm w^0)$  is a bounded constant. The first-order optimal condition of the algorithm in (\ref{oc2}) indicates that  $\nabla \mathcal{L}_{\tau,*}(\bm r^{k})  \ =  \bm d^{k}$.

(a)  From the expressions of ${P}_\lambda(|\bm \beta^k|)$ and $\nabla {P}_\lambda(|\bm \beta^k|)$, it can be seen that they are both positive and bounded functions.
Then $  \bm d^{k}$ and  ${P}_\lambda(|\bm \beta^k|)   - \frac{\|\bm d^k \|_2^2}{2\mu}$  are bounded.   Since $L_\mu(\bm w^0)$  is a bounded constant,    both $\mathcal{L}_{\tau,*}(\bm r^k) \ge 0$ and $\frac{\mu }{2}\left\|\bm{X\beta}^k + \bm{r}^k - \bm y  - \bm d^k/\mu \right\|_2^2 \ge 0$  are bounded.

Note that if $\| \bm r^k\|_2 \rightarrow \infty $, $\mathcal{L}_{\tau,*}(\bm r^k)  \rightarrow \infty$, then $ \bm r^k$  must be   bounded. Then, as long as  $\bm X^\top \bm X > \underline{\mu} \bm I_p $ , $\bm \beta^k$ must be bounded. Consequently,  $\bm w^k$ is   bounded. 

(b)    If    ${P}_\lambda(|\bm \beta^k|)$ is Mnet, Snet or Cnet,  then  ${P}_\lambda(|\bm \beta^k|) - \frac{\|\bm d^k \|_2^2}{2\mu} \to \infty$   as long as  $\| \bm \beta^k\|_2 \rightarrow \infty$. Hence, $ \bm \beta^k$ is  bounded.  For the same reason as  (a), we can also conclude that  both $\mathcal{L}_{\tau,*}(\bm r^k) \ge 0$ and $\frac{\mu }{2}\left\|\bm{X\beta}^k + \bm{r}^k - \bm y  - \bm d^k/\mu \right\|_2^2 \ge 0$  are bounded. $\mathcal{L}_{\tau,*}(\bm r^k) $ is   bounded, indicating that $\bm r^k$ is bounded.  When $ \bm \beta^k$, $\bm r^k$ and $\frac{\mu }{2}\left\|\bm{X\beta}^k + \bm{r}^k - \bm y  - \bm d^k/\mu \right\|_2^2$  are all bounded, it can be deduced that $\bm d^k$ is also bounded. Consequently,  $\bm w^k$ is   bounded.  $\blacksquare$ 
\end{proof}

Now, we derive the conclusions $\lim_{k \rightarrow +\infty} \|\bm w^k - \bm w^{k+1} \|_2 = 0$, which is similar to \cite{Guan2018An}. However, this conclusion only ensures that the difference between $\bm w^k$ and $\bm w^{k+1}$ is close to 0, but it does not guarantee that $\bm w^k$ converges.
\begin{lem}\label{lem4}
Let the sequence $\boldsymbol{w}^{k}=\{\boldsymbol{\beta}^k, \boldsymbol{r}^k, \boldsymbol{d}^k\}$ be generated by LADMM, then 
\begin{align*}
\sum_{k=0}^{\infty}\|\bm w^k - \bm w^{k+1} \|_2^2 < +\infty \ \text{and} \ \lim_{k \rightarrow +\infty} \|\bm w^k - \bm w^{k+1} \|_2 = 0.
\end{align*}
\end{lem}
\begin{proof}
From Lemma \ref{lem2} , we know  $\{L_\mu(\bm w^{k})\}_{k=1}^{+\infty}$ is a monotonically nondecreasing sequence. Next, we will determine its lower bound to demonstrate that  $\{L_\mu(\bm w^{k})\}_{k=1}^{+\infty}$ is convergent.

Since $\{\bm w^k\}$ is bounded, it has at least one cluster point. We assume $\bm w^*$ to be an arbitrary cluster point of $\{\bm w^k\}$. Let $\{\bm w^{k_j}\}$ represent the subsequence of $\{\bm w^k\}$ that converges to $\bm w^*$, i.e. $\bm w^{k_j} \rightarrow \bm w^*$. Since $L_\mu$ is a continuous function and $\{L_\mu(\bm w^k)\}_{k=1}^{+\infty}$  a monotonically nondecreasing sequence, then we have
$$L_\mu(\bm w^{k_j}) \ge L_\mu( \lim_{j  \rightarrow \infty}\bm w^{k_j})  = L_\mu( \bm w^*).$$
As a result,  $\{L_\mu(\bm w^{k_j})\}_{j=1}^{+\infty}$ has a lower bound, which, together with the fact that  $\{L_\mu(\bm w^{k_j})\}_{j=1}^{+\infty}$  is nonincreasing, means that   $\{L_\mu(\bm w^{k_j})\}_{j=1}^{+\infty}$  converges to $L_\mu( \bm w^*)$.  Considering that $\{L_\mu(\bm w^{k})\}_{k=1}^{+\infty}$ is monotonically nonincreasing, it follows that $L_\mu(\bm w^*)$ also serves as the lower bound for $\{L_\mu(\bm w^{k})\}_{k=1}^{+\infty}$.
Thus,  we can  conclude  that $\{L_\mu(\bm w^{k})\}_{k=1}^{+\infty}$ converges to   $L_\mu(\bm w^{*})$ and $L_\mu(\bm w^{k}) \ge L_\mu(\bm w^{*})$. 

Rearranging terms of Equation (\ref{key}) yields
\begin{align*}
 \left(  \frac{\mu}{2} - \frac{n}{\mu  \min\{c^2,\kappa^2 \}} \right)\|\bm r^k - \bm r^{k+1} \|_2^2 \le  L_\mu(\bm w^{k}) - L_\mu(\bm w^{k+1}),
\end{align*}
and summing up for $k=0, \dots, +\infty$, it follows
$$\left(  \frac{\mu}{2} - \frac{n}{\mu  \min\{c^2,\kappa^2 \}} \right)  \sum_{k=0}^{+\infty}\|\bm r^k - \bm r^{k+1} \|_2^2 \le   L_\mu(\bm w^{0}) -  L_\mu(\bm w^{*}) < +\infty. $$
Since $\left(  \frac{\mu}{2} - \frac{n}{\mu  \min\{c^2,\kappa^2 \}} \right)>0$, we obtain 
$\sum_{k=0}^{+\infty}\|\bm r^k - \bm r^{k+1} \|_2^2 < +\infty $. Consequently, it follows from Equation (\ref{rr}) that $\sum_{k=0}^{+\infty}\|\bm d^k - \bm d^{k+1} \|_2^2 < +\infty $. Hence, to complete the proof, we just need to prove that $\sum_{k=0}^{+\infty}\|\bm \beta^k - \bm \beta^{k+1} \|_2^2 < +\infty $.
For the $m$-th local machine, we have
\begin{equation*}
\left\{ \begin{array}{l}
\bm{d}_m^{k+1}=\bm{d}_m^{k}-\mu(\bm{X}_m\bm{\beta}^{k+1} + \bm{r}_m^{k+1} - \bm y_m),\\
\bm{d}_m^{k}=\bm{d}_m^{k-1}-\mu(\bm{X}_m\bm{\beta}^{k} + \bm{r}_m^{k} - \bm y_m),
\end{array} \right.
\end{equation*}
and make a difference between the above two equations to obtain
\begin{align*}
\bm{d}_m^{k+1} - \bm{d}_m^{k} = \bm{d}_m^{k} - \bm{d}_m^{k-1} - \mu \bm X_m(\bm \beta^{k+1} - \bm \beta^k) - \mu(\bm r_m^{k+1} - \bm r_m^k).
\end{align*}
It then follows that
\begin{align}\label{e41}
\| \mu \bm X_m(\bm \beta^{k+1} - \bm \beta^k) \|_2^2 &= \|(\bm{d}_m^{k} - \bm{d}_m^{k-1}) -  (\bm{d}_m^{k+1} - \bm{d}_m^{k}) - \mu(\bm r_m^{k+1} - \bm r_m^k)\|_2^2  \notag\\
& \le 3(\|\bm{d}_m^{k} - \bm{d}_m^{k-1} \|_2^2 + \|\bm{d}_m^{k+1} - \bm{d}_m^{k} \|_2^2 + \|\mu(\bm r_m^{k+1} - \bm r_m^k) \|_2^2). 
\end{align}
Since  $\bm X_m^\top \bm X_m > \underline{\mu} \bm I_p $,  we have
\begin{align}\label{e42}
\| \mu \bm X_m(\bm \beta^{k+1} - \bm \beta^k) \|_2^2 \ge \mu^2  \underline{\mu} \|\bm \beta^{k+1} - \bm \beta^k \|_2^2.
\end{align}
Then, combining Equations (\ref{e41}) and (\ref{e42}), we  obtain
$$ \sum_{k=0}^{+\infty}\|\bm \beta^k - \bm \beta^{k+1} \|_2^2 < +\infty.  $$
Since $\boldsymbol{w}^{k}=\{\boldsymbol{\beta}^k, \boldsymbol{r}^k, \boldsymbol{d}^k\}$, we have established that 
$$\sum_{k=1}^{\infty}\|\bm w^k - \bm w^{k+1} \|_2^2 \le \infty,$$
which indicates that  $\lim_{k \rightarrow +\infty} \|\bm w^k - \bm w^{k+1} \|_2 = 0$.   $\blacksquare$ 
\end{proof}
 From Lemma \ref{lem2}, we know that  $\bm w^k$ has at least one limit point. Next, we will prove some properties about the limit points.
\begin{lem}\label{lem5}
Let  $S(\bm w^ \infty)$ denote the set of the limit points of $\{\bm w^k\}$, and $C(\bm w^*)$  denote  the set of  critical points of $L_\mu(\bm w)$. $S(\bm w^ \infty) \subset C(\bm w^*)$, and $L_\mu(\bm w)$ is finite and constant on $S(\bm w^ \infty)$.  
\end{lem}
\begin{proof}
We say $\bm w^*$ is the critical point  of $L_\mu(\bm w)$, if it satisfies
\begin{equation}\label{kkt}
\left\{ \begin{array}{l}
\bm X^\top \bm d^* \ \in \partial{P}_\lambda(|\bm \beta^{*}|),\\
\bm d^* \ = \nabla \mathcal{L}_{\tau,*}(\bm r^{*}) ,\\
\bm 0= \ \bm{X}\bm{\beta}^{*} + \bm{r}^{*} - \bm y.
\end{array} \right.
\end{equation}
Recalling the  optimality conditions for the $k+1$-iteration of LADMM in (\ref{oc}),  we get
\begin{equation}\label{occ}
\left\{ \begin{array}{l}
\bm 0 \ \in \partial{P}_\lambda(|\bm \beta^{k+1}|) - \bm X^\top \bm d^{k+1} + \bm S(\bm \beta^{k+1} - \bm \beta^k),\\
\bm 0 \ = \nabla \mathcal{L}_{\tau,*}(\bm r^{k+1}) - \bm d^{k+1} ,\\
\bm{d}^{k+1}=\bm{d}^{k}-\mu(\bm{X}\bm{\beta}^{k+1} + \bm{r}^{k+1} - \bm y).
\end{array} \right.
\end{equation}
For any limit point  $\bm w^ \infty$, from Lemma \ref{lem4} and (\ref{occ}), we have
\begin{equation}\label{occ2}
\left\{ \begin{array}{l}
\bm X^\top \bm d^\infty \ \in \partial{P}_\lambda(|\bm \beta^\infty|),\\
\bm d^\infty \ = \nabla \mathcal{L}_{\tau,*}(\bm r^\infty),\\
\bm{0} =\bm{X}\bm{\beta}^\infty + \bm{r}^\infty - \bm y.
\end{array} \right.
\end{equation}
As a consequence,  $ S(\bm w^ \infty) \subset S(\bm w^*)$. Note that  $L_\mu(\bm w^{k})$ is nonincreasing and convergent, it follows that  $L_\mu(\bm w)$ is finite and constant on $S(\bm w^ \infty)$. $\blacksquare$ 
\end{proof}

Now, we are ready to prove Theorem \ref{TH2}. Despite the LADMM algorithm having an additional term of $\frac{1}{2}\left \| \bm \beta - \bm \beta^k \right \|_{\bm S}^2$ compared to the ADMM algorithm in \cite{Guo2016Convergence}, the convergence proof is similar to that of Theorem 3.1 in \cite{Guo2016Convergence}.
Following the same steps as their proof, we  can conclude that 
\begin{align}
\sum_{k=0}^{+\infty}\|\bm w^{k+1} - \bm w^k  \|_2 < + \infty.
\end{align}
That is, $\{\bm w^k\}$  is a Cauchy sequence and it converges. Together with Lemma \ref{lem5}, $\{\bm w^k\}$ converges to a critical point of $L_\mu(\bm w)$.

\section{Supplementary Numerical Experiments}\label{C}
\subsection{Supplementary Experiments for Section \ref{sec41}}\label{C.1}
In this subsection, we analyze the performance of LADMM in comparison to other ADMM variants when solving  quantile regression with the MCP ($a=3$) penalty. Table \ref{tab6} presents the results for the nonparallel version ($M=1$), while Table \ref{tab7} displays the results for the parallel version ($M>1$). The results of Table \ref{tab6} and Table \ref{tab7} demonstrate that our proposed LADMM algorithm  has significant advantages over existing algorithms when solving quantile regressions with MCP.
\begin{table}[H]
\caption{\footnotesize{Comparison of various ADMMs with MCP penalty of different data sizes.}}\label{tab6}
\centering
\resizebox{1\columnwidth}{!}{
\begin{threeparttable}
\begin{tabular}{lllllllllll}
\toprule[1.5pt]
$(n,p)$ & \multicolumn{5}{l}{(30000,1000)} & \multicolumn{5}{l}{(1000,30000)} \\ 
\cmidrule(lr){1-1}\cmidrule(lr){2-6}\cmidrule(lr){7-11}
$\tau=0.3$ & P1 & P2 & AE & Ite & Time & P1 & P2 & AE & Ite & Time \\ 
\midrule[1pt]
QPADM & 100 & 100 & \bf 0.013(0.01) & 95.9(9.8) & 34.67(5.7) & 0 & 100 & 6.68(0.97) & 500(0.0) & 2013.4(59.4)\\ 
QPADMslack & 100 & 100 & 0.042(0.02) & 47.2(6.1) & 38.72(5.3) & 0 & 100 & 7.13(0.89) & 226(13.8) & 1234.8(36.7) \\ 
FSLADMM & 100 & 100 & 0.015(0.03) & 500(0.0) & 161.32(10.9) & 100 & 100 & 0.69(0.07) & 500(0.0) & 129.5(11.3)\\ 
LADMM & 100 & 100 & 0.014(0.01) & \bf 39.5 (3.0) & \bf 31.87(3.9) & 100 & 100 & \bf 0.38(0.05) & \bf 123(10.2) & \bf 56.3(4.6) \\ 
\midrule[1pt]
$\tau=0.5$ & P1 & P2 & AE & Ite & Time & P1 & P2 & AE & Ite &Time \\ 
\midrule[1pt]
QPADM & 0 & 100 &\bf 0.011(0.01) & 87.9(8.5) & 29.58(5.1) & 0 & 100 & 5.92(0.82) & 500(0.0) & 1967.5(60.1)\\ 
QPADMslack & 0 & 100 & 0.038(0.03) & 41.7(5.5) & 32.39(4.8) & 0 & 100 & 7.24(0.78) & 232(12.7)& 1094.7(40.2) \\ 
FSLADMM & 0 & 100 & 0.015(0.04) & 500(0.0) & 148.95(9.7) & 0 & 100 & 0.64(0.06) & 500(0.0) & 131.8(12.5)\\ 
LADMM & 0 & 100 & 0.012(0.01) & \bf 35.2 (2.7)& \bf 28.59(3.5) & 0 & 100 & \bf 0.35(0.03) & \bf 119(9.1) & \bf 57.4(5.1)\\
\midrule[1pt]
$\tau=0.7$ & P1 & P2 & AE & Ite & Time & P1 & P2 & AE & Ite &Time \\ 
\midrule[1pt]
QPADM & 100 & 100 & \bf 0.016(0.02) & 99.2(9.3) & 38.54(6.1) & 100 & 100 & 6.81(0.77) & 500(0.0) & 2123.7(58.9) \\ 
QPADMslack & 100 & 100 & 0.049(0.05) & 48.5(5.7) & 42.21(5.6) & 100 & 100 & 6.93(0.72) & 262(15.7)& 1182.6(49.3) \\ 
FSLADMM & 100 & 100 & 0.021(0.04) & 500(0.0) & 157.17(9.9) & 100 & 100 & 0.72(0.06) & 500(0.0) & 136.5(13.0) \\ 
LADMM & 100 & 100 & 0.018(0.01) & \bf 40.8 (3.1)& \bf 30.27(3.2) & 100 & 100 &\bf  0.39(0.03) & \bf 130(9.9) & \bf 60.2(4.9) \\
\toprule[1.5pt]
\end{tabular}
\begin{tablenotes}
        \footnotesize
        \item[*] The meanings of the notations used in this table are as follows: P1 (\%): proportion that $x_1$ is selected; P2 (\%): proportion that $x_6$, $x_{12}$, $x_{15}$, and $x_{20}$ are selected; AE: absolute estimation error; Ite: number of iterations; Time (s): CPU running time. Numbers in the parentheses represent the corresponding standard deviations, and the optimal solution is represented in bold.
\end{tablenotes}
\end{threeparttable}}
\end{table}

\begin{table}[H]
\caption{\footnotesize{Comparison of various parallel ADMMs with MCP penalty.}}\label{tab7}
\centering
\resizebox{1\columnwidth}{!}{
\begin{threeparttable}
\begin{tabular}{lllllllll}
\toprule[1.5pt]
$(n=500000,p=1000)$ & \multicolumn{4}{l}{$\tau=0.4$} & \multicolumn{4}{l}{$\tau=0.5$} \\ 
\cmidrule(lr){1-1}\cmidrule(lr){2-5}\cmidrule(lr){6-9}
$M=5$ & Nonzero & AE & Ite & Time & Nonzero & AE & Ite & Time \\ 
\midrule[1pt]
QPADM & 30.8(2.8) & 0.048(0.005) & 492.9(39.6) & 204.7(25.7) & 25.7(2.2) & 0.043(0.004) & 495.9(39.7) & 212.6(30.2)\\ 
QPADMslack & 28.1(3.1) & 0.045(0.006) & 335.1(26.8) & 127.2(16.5) & 24.1(1.8) & 0.042(0.005) & 347.2(45.1) & 135.2(15.9) \\ 
LADMM &\bf 5.3(0.3) &\bf 0.029(0.001) & \bf 79.5 (4.2) & \bf 41.2(4.1) & \bf 4.1(0.07) & \bf 0.024(0.001) &\bf  82.6 (5.3) & \bf 39.7(3.8) \\ 
\midrule[1pt]
$M=20$ & Nonzero & AE & Ite & Time & Nonzero & AE & Ite & Time \\ 
\midrule[1pt]
QPADM & 32.6(3.5) & 0.051(0.06) & 500(0.0) & 129.8(12.6) & 28.3(2.4) & 0.048(0.05) & 500(0.0) & 129.6(15.0)\\ 
QPADMslack & 29.4(3.4) & 0.048(0.07) & 341.4(35.2) & 62.9(8.5) & 25.9(2.1) & 0.047(0.05) & 402.3(51.5) & 72.9(7.8) \\ 
LADMM & \bf 5.3(0.3) &\bf  0.029(0.001) & \bf 80.4 (4.4) &\bf 28.5(3.2) &\bf 4.0(0.07) & \bf 0.024(0.001) &\bf 83.3 (5.4) &\bf 27.5(3.5)\\
\midrule[1pt]
$M=100$ & Nonzero & AE & Ite & Time & Nonzero & AE & Ite & Time \\ 
\midrule[1pt]
QPADM & 35.1(4.0) & 0.076(0.08) & 500(0.0) & 48.54(4.1) & 30.8(2.7) & 0.052(0.06) & 500(0.0) & 38.4(3.9) \\ 
QPADMslack & 32.7(3.8) & 0.069(0.06) & 448.5(56.2) & \bf 18.21(2.5) & 28.6(3.2) & 0.049(0.06) & 441.4(62.1) & \bf 20.1(2.5) \\ 
LADMM &\bf 5.3(0.3) & \bf 0.030(0.001) & \bf 84.5 (4.7) & 19.27(2.6) & \bf 4.1(0.07) &\bf 0.025(0.001) & \bf 84.9 (5.6) & 20.7(2.2) \\
\toprule[1.5pt]
\end{tabular}
\begin{tablenotes}
        \footnotesize
        \item[*] P1 and P2 are not presented in Table  \ref{tab7} because all methods have a value of 100 for these two metrics. 
The Nonzero, AE, and Ite of LADMM are not greatly affected by the $M$ value. 
\end{tablenotes}
\end{threeparttable}}
\end{table}

\subsection{Supplementary Experiments for Section \ref{sec42}}\label{C.2}
In this subsection, we compare the performance of LADMM and ILAMM in solving least squares regressions with MCP ($a=3$) and Capped-$\ell_1$ ($a=3$) when the errors follow the Lognormal distribution. Table  \ref{tab8} shows that our algorithm performs better than ILAMM in solving these two regression models. 
\begin{table}[H]
\caption{\footnotesize{Comparison of LS-MCP and LS-Capped-$\ell_1$ Calculated by LADMM and ILAMM}}\label{tab8}
\centering
\resizebox{1\columnwidth}{!}{
\begin{threeparttable}
\begin{tabular}{llllllllll}
\toprule[1.5pt]
& & \multicolumn{4}{l}{MCP} & \multicolumn{4}{l}{Capped-$\ell_1$} \\ 
\cmidrule(lr){3-6}\cmidrule(lr){7-10}
Algorithm & $(n,p)$ & FP & FN & AE & Time & FP & FN & AE & Time \\ 
\midrule[1pt]
\multirow{4}*{ILAMM} & (100,1000) & 0 & 1 & \bf 0.43(0.08) & 8.6(0.9) & 0 & 1 & 0.52(0.07) & 9.1(1.0) \\
& (100,10000) & 0 & 2 & 0.61(0.10) & 164.2(10.6) & 0 & 2 & 0.67(0.12) & 157.3(9.8) \\
& (1000,10000) & 0 & 2 & 0.53(0.08) & 252.5(21.9) & 0 & 2 & 0.58(0.09) & 283.8(27.5) \\
& (10000,10000) & 0 & 0 & 0.09(0.02) & 1846.8(138.6) & 0 & 0 & 0.13(0.04) & 2014.5(142.4) \\
\midrule[1pt]
Algorithm & $(n,p)$ & FP & FN & AE & Time & FP & FN & AE & Time \\ 
\midrule[1pt]
\multirow{4}*{LADMM} & (100,1000) & 0 & 0 & 0.44(0.04) & \bf 1.4(0.07) & 0 & 0 & \bf 0.39(0.02) & \bf 1.2(0.09) \\
& (100,10000) & 0 & 0 & \bf 0.52(0.05) & \bf 7.6(1.12) & 0 & 0 & \bf 0.45(0.04) & \bf 5.5(1.01) \\
& (1000,10000) & 0 & 0 &\bf 0.48(0.07) &\bf 15.8(2.65) & 0 & 0 & \bf 0.41(0.05) &\bf  13.1(1.98) \\
& (10000,10000) & 0 & 0 & \bf 0.05(0.01) &\bf 35.7(5.48) & 0 & 0 &\bf 0.04(0.01) & \bf 25.5(4.70) \\
\toprule[1.5pt]
\end{tabular}
\begin{tablenotes}
        \footnotesize
        \item[*] False Positive (FP) refers to the number of variables with a coefficient of zero that are mistakenly included in the final model, while False Negative (FN) refers to the number of variables with non-zero coefficients that are omitted from the model. 
\end{tablenotes}
\end{threeparttable}}
\end{table}

\subsection{Supplementary Experiments for Section \ref{sec43}}\label{C.3}
In this subsection, we compare the performance of LADMM and cdaSQR in solving smooth quantile regressions with Mnet ($a=3$) penalty. Table \ref{tab9} demonstrates the results. However, cdaSQR is unable to solve the Cnet ($a=3$) penalty. Therefore, we solely present the results of LADMM in solving Cnet regression in Figure \ref{fig6}. Both the table and the figure illustrate the superiority of our algorithm.
\begin{table}[H]
\caption{\footnotesize{Comparison of  LADMM and cdaSQR with $\delta=0.5$ and $\tau=0.5$}}\label{tab9}
\centering
\resizebox{1\columnwidth}{!}{
\begin{threeparttable}
\begin{tabular}{llllllllll}
\toprule[1.5pt]
& & \multicolumn{4}{l}{cdaSQR} & \multicolumn{4}{l}{LADMM} \\ 
\cmidrule(lr){3-6}\cmidrule(lr){7-10}
Loss & $(n,p)$ & FP & FN & AE & Time & FP & FN & AE & Time \\ 
\midrule[1pt]
\multirow{3}*{$\mathcal{L}_{\tau,c}$ } & (5000,500) & 1.96 & 6.48 & \bf 0.36(0.03) & \bf 3.46(0.23) &\bf 0 & \bf 0.08 & 0.37(0.02) & 3.72(0.29) \\
& (10000,1000) & 2.10 & 7.21 & 0.48(0.05) & 5.13(0.61) & \bf 0 &\bf 0.09 &\bf 0.45(0.04) &\bf 4.85(0.55) \\
& (20000,2000) & 3.08 & 9.59 & 0.62(0.07) & 24.36(3.28) & \bf 0 & \bf 0.10 & \bf 0.53(0.06) &\bf 19.82(2.7) \\
\midrule[1pt]
Loss & $(n,p)$ & FP & FN & AE & Time & FP & FN & AE & Time \\ 
\midrule[1pt]
\multirow{3}*{$\mathcal{L}_{\tau,\kappa}$ } & (5000,500) & 2.01 &  6.07 & \bf 0.39(0.02) & 4.12(0.32) & \bf 0 & \bf 0.05 & 0.42(0.03) &\bf 3.64(0.32) \\
& (10000,1000) & 2.54 & 7.42 & 0.51(0.06) & 6.27(0.95) & \bf 0 & \bf 0.06 & \bf 0.46(0.05) & \bf 4.58(0.57) \\
& (20000,2000) & 2.92 & 9.43 & 0.64(0.08) & 26.72(3.73) & \bf 0 & \bf 0.08 & \bf 0.57(0.06) & \bf 21.36(2.9) \\
\toprule[1.5pt]
\end{tabular}
\end{threeparttable}}
\end{table}

\begin{figure}[htbp]
\centering
\includegraphics[width=18cm]{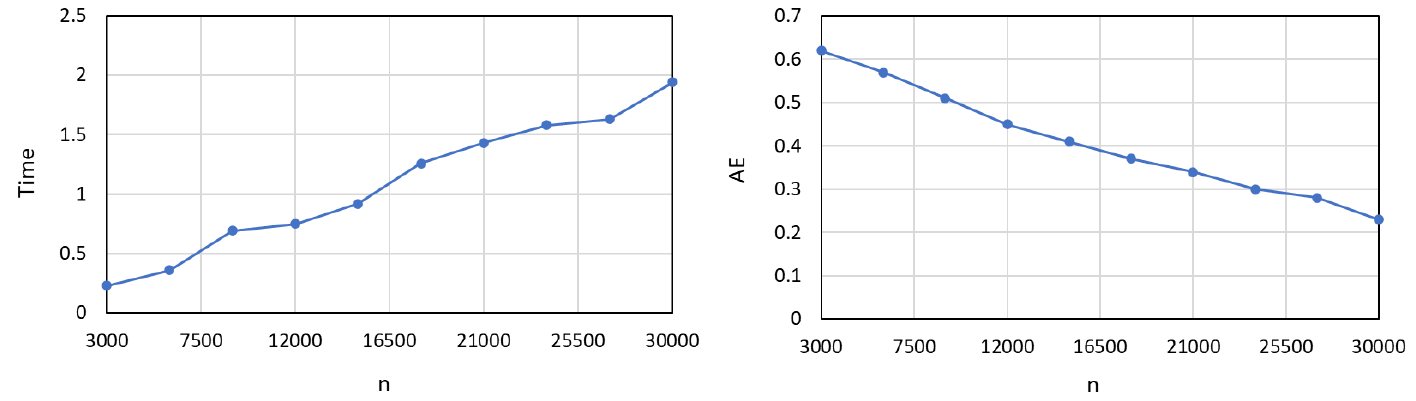}
\caption{\footnotesize{Variation of computing time and  AE with $n$}}\label{fig6}
\end{figure}

\newpage
\subsection{Supplementary Experiments for Section \ref{sec5}}\label{C.4}
In this subsection, we compare the performance of several parallel ADMM algorithms in solving MCP ($a=3$) regressions using an online publicly available dataset. The results ($M=100$) are presented in Table \ref{tab10}. The prediction error (PE) is calculated as follows: $\text{PE} = \frac{1}{n_{test}} \sum_{i=1}^{n_{test}} |y_i - \hat{y}_i|$, where $n_{test} = 4644$ represents the size of the test set. The experimental results presented in Table \ref{tab10} indicate that our LADMM algorithm exhibits competitive performance compared to other parallel algorithms when solving MCP ($a=3$) regressions on this real dataset.
\begin{table}[H]
\caption{\footnotesize{Analysis of the news popularity data under the MCP penalty.}}\label{tab10}
\centering
\resizebox{1\columnwidth}{!}{
\begin{threeparttable}
\begin{tabular}{llllllll}
\toprule[1.5pt]
& Algorithm & $x_{14}$ & $x_{27}$ & Nonzero & PE & Ite & Time \\ 
\midrule[1pt]
\multirow{3}*{$\tau=0.5$} & QRADM & -0.023(0.003) & 0.056(0.004) & 33.9(2.30) & 0.20(0.01) & 230.2(44.3) & 3.62(0.48)\\
& QRADMslack & -0.020(0.001) & 0.052(0.002) & 35.1(1.95) & 0.20(0.01) & 118.3(10.6) & 2.01(0.36)\\
& $\text{LADMM}_{\rho}$ & -0.011(0.001) & 0.024(0.001) & 12.3(0.66) & \bf 0.18(0.01) & 73.6(5.9) & \bf 1.11(0.21)\\
& $\text{LADMM}_{\mathcal L_{\tau,c},\delta=1}$ & -0.013(0.001) & 0.026(0.002) & 13.1(0.56) & 0.19(0.01) & \bf 71.5(6.1) & 1.13(0.19)\\
& $\text{LADMM}_{\mathcal L_{\tau,\kappa},\delta=1}$ & -0.011(0.001) & 0.023(0.001) & 13.8(0.72) & 0.19(0.01) & 80.2(6.3) & 1.12(0.22)\\
\midrule[1pt]
& Algorithm & $x_{14}$ & $x_{27}$ & Nonzero & FP & Ite & Time \\ 
\midrule[1pt]
\multirow{3}*{$\tau=0.9$} & QRADM & -0.138(0.013) & 0.317(0.009) & 38.9(2.92) & 0.24(0.01) & 500(0.00) & 3.93(0.52)\\
& QRADMslack & -0.120(0.009) & 0.289(0.007) & 40.1(1.85) & 0.24(0.01) & 472.3(51.6)& 2.56(0.37)\\
& $\text{LADMM}_{\rho}$ & -0.043(0.002) & 0.123(0.004) & 13.6(0.71) & \bf 0.21(0.01) & \bf 66.8(7.2) & 1.23(0.21)\\
& $\text{LADMM}_{\mathcal L_{\tau,c},\delta=1}$ & -0.050(0.002) & 0.126(0.004) & 13.9(0.69) & 0.23(0.01) & 73.5(8.3) & \bf 0.95(0.21)\\
& $\text{LADMM}_{\mathcal L_{\tau,\kappa},\delta=1}$ & -0.053(0.003) & 0.128(0.005) & 14.2(0.80) & 0.23(0.01) & 76.1(8.6) & 0.23(0.25)\\
\toprule[1.5pt]
\end{tabular}

\end{threeparttable}}
\end{table}

\end{document}